\documentclass[10pt]{article}

\pdfoutput=1

\usepackage{customizations}

\begin{document}

\title{Interpolation and vector bundles on curves}
\author{Atanas Atanasov}
\date{}

\maketitle

\begin{abstract}
  We define several notions of interpolation for vector bundles on curves and discuss their relation to slope stability. The heart of the paper demonstrates how to use degeneration arguments to prove interpolation. We use these ideas to show that a general connected space curve of degree $d$ and genus $g$ satisfies interpolation for $d \geq g+3$ unless $d = 5$ and $g = 2$. As a second application, we show that a general elliptic curve of degree $d$ in $\Pbb^n$ satisfies weak interpolation when $d \geq 7$, $d \geq n+1$, and the remainder of $2d$ modulo $n-1$ lies between $3$ and $n-2$ inclusive. Finally, we prove that interpolation is equivalent to the---a priori stricter---notion of strong interpolation. This is useful if we are interested in incidence conditions given by higher dimensional linear spaces.  
\end{abstract}


\tableofcontents


\section{Introduction}
\label{S:introduction}

The study of moduli spaces of curves is a very rich branch of algebraic geometry. There is a long history of interaction between curves in the abstract (e.g., the Deligne-Mumford compactification $\overline{\Mc}_g$) and concrete realizations of curves (e.g., Hilbert schemes and Kontsevich spaces). Understanding of the geometry on one side could lead to a result on the other and vice versa. For example, many results about the birational geometry of $\overline{\Mc}_g$ use trigonal curves or curves on K3s and other surfaces. Interpolation properties help us understand subvarieties of the Hilbert scheme given by incidence conditions. We will use this section to give a more precise meaning of this statement and explain how these ideas can eventually be used to study the birational geometry of $\overline{\Mc}_g$.

Our study of interpolation was inspired by an interest in constructing moving curves in $\overline{\Mc}_g$. Harris-Morrison \cite{Harris-Morrison-slopes} produced such curves by studying ramified covers over $\Pbb^1$ and used them to bound the slopes of effective divisors on $\overline{\Mc}_g$. Fedorchuk \cite{Fedorchuk-thesis} continued this line of work by using curves in $\Pbb^2$. A thorough overview of this topic can be found in \cite{Chen-Farkas-Morrison}. Instead of attacking the construction question directly, our goal is to create a broader framework for studying the deformation theory of curves in projective space. Constructing moving curves is only one of several possible applications.

We proceed to outline how Hilbert schemes of curves could be used to produce moving curves in $\overline{\Mc}_g$ and motivate the definition of interpolation. The Hilbert scheme $\Hc_{d,g,n}$ parametrizes degree $d$ and arithmetic genus $g$ curves in $\Pbb^n$. The restricted Hilbert scheme $\Ic_{d,g,n}$ is defined as the closure of the locus of smooth irreducible curves in $\Hc_{d,g,n}$. For every triple $(d,g,n)$ there is a rational map $\phi_{d,g,n} \cn \Ic_{d,g,n} \drarr \overline{\Mc}_g$ from the restricted Hilbert scheme to the Deligne-Mumford moduli space of curves. Given a fixed $g$, it is always possible to find $d$ and $n$ such that $\phi_{d,g,n}$ is dominant. If we can construct a moving curve in $\Ic_{d,g,n}$ which does not map to a point, then it maps to a moving curve in $\overline{\Mc}_g$.

Let $C \subset \Pbb^n$ be a smooth irreducible curve and $[C] \in \Ic_{d,g,n}$ the corresponding point. The tangent space $T_{[C]} \Ic_{d,g,n}$ can be identified with $\H^0(N_C)$. If $\H^1(N_C) = 0$, then $[C]$ is a smooth point of $\Ic_{d,g,n}$. Pick a point $p \in C$ and consider the subvariety $\Jc(p) \subset \Ic_{d,g,n}$ consisting of all curves containing $p$. By construction $[C] \in \Jc(p)$ and we can identify the tangent space $T_{[C]}\Jc(p)$ with $\H^0(N_C(-p))$. Since the normal bundle $N_C$ has rank $n-1$, the expectation is that $\Jc(p)$ has codimension $n-1$ in $\Ic_{d,g,n}$. After choosing a second point $q \in C \setminus \{p\}$, we can define analogously the variety $\Jc(p,q)$ which satisfies $T_{[C]} \Jc(p,q) \cong \H^0(N_C(-p-q))$ and the process continues. The idea is to use just enough points so that $\Jc(p_1, \dots, p_m)$ is a curve. By allowing the $p_i$ to vary, we have constructed a moving curve in $\Ic_{d,g,n}$.

While very instructive, there are several issues with this naive idea. First, we would expect $\Jc(p_i)$ to be a curve only if $\H^0(N_C) - 1$ is a multiple of $n-1$. Fortunately, there is a simple generalization which circumvents this restriction. Suppose $L \subset \Pbb^n$ is a linear space which intersects $C$ transversely in a point $p$. Consider the space $\Jc(L)$ of curves incident to $L$, that is, they intersect $L$ nontrivially. The tangent space $T_{[C]} \Jc(L)$ can be identified with $\H^0(N')$ where $N'$ is a vector bundle defined by the short exact sequence
\[\xymatrix{
  0 \ar[r] &
  N' \ar[r] &
  N_C \ar[r] &
  N_{C,p} / [T_p L] \ar[r] &
  0.
}\]
The third term is the quotient of $N_{C, p}$ by $(T_p L + T_p C)/T_p C$, the image of $T_p L$ in $N_{C,p}$. We expect that $\Jc(L)$ has codimension
\[
\dim N_{C,p}/[T_pL] =
n - 1 - \dim L.
\]
The construction extends to multiple linear spaces $L_1, \dots, L_m$ and the tangent space to $\Jc(L_1, \dots, L_m)$ at $[C]$ is controlled by a vector bundle $N' \subset N_C$ whose cokernel is supported on $C \cap \bigcup_i L_i$. We can always choose multiple points and a single higher dimensional linear space such that the incident curves form a variety whose expected dimension is $1$.

The second, and more pressing, issue lies with the difference between expectation and reality. Computing the actual dimension of a variety $\Jc(L_1, \dots, L_m)$ often requires some significant understanding of the curves we are dealing with and their deformation theory. Intuitively speaking, the notion of interpolation formalizes the expected behavior we have been discussing.

Given a triple of integers $d$, $g$ and $n$, we ask whether the general smooth curve in $\Ic_{d,g,n}$ satisfies interpolation. The goal of this paper is to develop the notion of interpolation rigorously and demonstrate how to use degeneration techniques in order to approach our question for certain triples $(d,g,n)$. For example, the union of a rational curve and a secant line can be treated as an elliptic curve. This degeneration plays a key role in our proof that a general elliptic curve with $d \geq n + 1$ and $d \geq 7$ satisfies interpolation under an additional congruence condition. More generally, the addition of $g$ secant lines to a rational curve produces a genus $g$ curve. We combine this idea with a careful analysis of normal bundles to show that a general smooth curve in $\Pbb^3$ satisfies interpolation as long as $d \geq g + 3$ and $(d,g) \neq (5,2)$.

\begin{named}{Summary}
  \Cref{S:elementary-modifications} introduces elementary modifications which are later used in the definition of interpolation. In \cref{S:interpolation-properties}, we discuss the relation between interpolation, slope stability, and section stability. We illustrate these notions using vector bundles on the projective line. Once the basic language is established, we discuss the use of degenerate curves in interpolation arguments. \Cref{S:nodal-curves} establishes two results which respectively allow us to attach two curves in a single point and attach a secant line to a curve. We show interpolation is an open property and make other remarks about Hilbert schemes in \cref{S:Hilbert-schemes}. The framework we constructed is applied to elliptic curves in \cref{S:elliptic-curves} and curves in $\Pbb^3$ in \cref{S:curves-in-P3}. The birational geometry of Hirzebruch surfaces plays a central role in our study of space curves. Finally, we prove the equivalence of strong and regular interpolation in section \cref{S:strong-regular}.
\end{named}

\begin{named}{Conventions}
  Unless otherwise noted, we will consistently make the following conventions.
  \begin{itemize}
  \item
    We will work over an algebraically closed field $K$ of characteristic $0$.
  \item
    All varieties are reduced, separated, finite type schemes over $K$.
  \item
    All curves are connected and locally complete intersection (lci); all families of curves have connected lci fibers.
  \item
    All vector bundles are locally free sheaves of finite constant rank.
  \item
    A subbundle refers to a vector subbundle with locally free quotient.
  \item
    All divisors are Cartier.
  \item
    We will call a vector bundle nonspecial if it has no higher cohomology.
  \end{itemize}
\end{named}

\begin{named}{Acknowledgments}
  First and foremost, I would like to thank Joe Harris who introduced me to the unparalleled world of algebraic geometry. He suggested a number of guiding questions which led to this work.

  I am grateful to Gabriel Bujokas, Alex Perry, and Eric Larson who read an early draft, and provided numerous insightful comments. Thanks also go to Edoardo Ballico who made several constructive suggestions after reading a previous version of the paper.

  I would also like to acknowledge the generous support of National Science Foundation grant DMS-1308244, for ``Nonlinear Analysis on Sympletic, Complex Manifolds, General Relativity, and Graphs.''
\end{named}


\section{Elementary modifications of vector bundles}
\label{S:elementary-modifications}

Let $E$ be a vector bundle on an algebraic variety $X$. Consider a quotient of the form
\[\xymatrix{
  E \ar[r]^-{q} &
  Q \ar[r] &
  0
}\]
where $Q$ is a locally free $\Oc_D$-sheaf for a Cartier divisor $D \subset X$.  We will refer to the kernel of $q$ as the \emph{(elementary) modification of $E$ corresponding to $q$} and denote it by $M(E, q)$. It can be shown that all such modifications are vector bundles.

While the definition of an elementary modification is quite general, the applications we care about assume $X = C$ is a curve and $D = \{ p_1, \dots, p_m \}$ is a reduced Cartier divisor. If we take $Q = E|_D = \bigoplus_i E_{p_i}$, then
\[
M(E, q) = E(-D).
\]
Following the long exact sequence in cohomology, the global sections of this bundle admit a convenient presentation in terms of the global sections of $E$:
\begin{align*}
  \H^0(E(-D))
  &= \{ \sigma \in \H^0(E) \;|\; \sigma|_D = 0 \} \\
  &= \{ \sigma \in \H^0(E) \;|\; \sigma(p_i) = 0 \textrm{ for all $i$} \}.
\end{align*}

In a generalization of the above example, we keep $D = \{ p_1, \dots, p_m \}$ and consider a subspace $V_i \subset E_{p_i}$ for each $1 \leq i \leq m$. We can assemble a sheaf by taking the quotient by $V_i$ at each of the corresponding fibers:
\[
Q = \bigoplus_i E_{p_i} / V_i.
\]
By composing the evaluation at $D$ with the natural quotient morphism
\[\xymatrix{
  E \ar[r] &
  \bigoplus_i E_{p_i} \ar[r] &
  Q = \bigoplus_i E_{p_i}/V_i,
}\]
we construct a modification which we will denote by $M(E, V_i) = M(E, V_1, \dots, V_m)$. The points $p_i$ are implicitly understood. Its global sections admit an analogous presentation:
\begin{align*}
  \H^0(M(E, V_i))
  &= \{ \sigma \in \H^0(E) \;|\; \sigma(p_i) \in V_i \textrm{ for all $i$} \}.
\end{align*}
Setting all $V_i = 0$ recovers $E(-D)$ as above.

\begin{remark}
  \label{T:Cartier-smooth-locus}
  Note that we require the divisor $D \subset C$ to be reduced and Cartier. This means that $D$ is contained in the smooth locus $C_\sm$ of $C$.
\end{remark}


\section{Interpolation and related properties}
\label{S:interpolation-properties}

Let $C$ be a curve and $E$ a nonspecial vector bundle on it. Consider a point $p \in C_\sm$ and a subspace of the fiber $V \subset E_p$. Comparing $\H^0(M(E, V))$ with $\H^0(E)$, we expect the dimension of the former is $\codim(V, E_p)$ less than the dimension of the latter. However, this is not always the case. For example, if $\H^0(E) = 0$, it is impossible to have $\h^0(M(E, V)) < 0$. More generally, one has to look at the long exact sequence
\[\xymatrix{
  0 \ar[r] &
  \H^0(M(E, V)) \ar[r] &
  \H^0(E) \ar[r] &
  E_p/V \ar[r] &
  \H^1(M(E, V)) \ar[r] &
  0.
}\]

We are interested in studying the conditions under which the expected dimension is achieved. It may be the case that a specific pair $(p, V)$ fails to satisfy the proposed condition but that behavior is isolated. It is more meaningful to ask the same question for a general point $p \in C_\sm$ and a general subspace $V \subset E_p$ of a fixed dimension. The notion of interpolation formalizes this idea.

Fix a curve $C$ and a nonspecial vector bundle $E$ on it. Let
\[
\lambda = (\lambda_1, \dots, \lambda_m)
\]
be a weakly decreasing sequence of integers such that $0 \leq \lambda_i \leq \rank E$ for all $i$, and $\sum_i \lambda_i \leq \h^0(E)$. Let $\Delta \subset C_\sm^m$ denote the big diagonal. A point of the Grassmannian bundle
\[\xymatrix{
  \prod_i \Gr(\rank E - \lambda_i, E) \ar[d] \\
  C_\sm^m \setminus \Delta
}\]
corresponds to the datum
\[
( p_i \in C_\sm, V_i \in \Gr(\rank E - \lambda_i, E_{p_i}) )
\]
such that all $p_i$ are distinct. Given this information, we can construct an elementary modification $M(E, V_i)$.

\begin{definition}[$\lambda$-interpolation]
  \label{T:interpolation-lambda}
  We say $E$ satisfies \emph{$\lambda$-interpolation} if it is nonspecial and there exists a point on the above Grassmannian bundle such that the corresponding elementary modification $E'$ satisfies
  \[
  \h^0(E) - \h^0(E') = \sum_i \lambda_i.
  \]
\end{definition}

\begin{remark}
  \label{T:smooth-components}
  Note that the definition of interpolation requires the inequality to hold for only a single point of the Grassmannian bundle. By upper semi-continuity of $\h^0$, the notion of $\lambda$-interpolation is not strengthened if we require the equality $\h^0(E) - \h^0(E') = \sum_i \lambda_i$ to hold for a general point of one of the components of the Grassmannian bundle.

  The (connected) components of the Grassmannian bundle correspond to (connected) components of the base $C_\sm^m \setminus \Delta$. If the curve $C$ is reducible, then $C_\sm$ is not connected.
\end{remark}

\begin{definition}[Interpolation properties]
  \label{T:interpolation-properties}
  Let us write $\h^0(E) = q \cdot \rank E + r$ where $0 \leq r < \rank E$. The following table describes, in increasing strength, all interpolation-type properties we will use.
  \begin{center}
    \begin{tabular}{ll}
      \toprule
      $\lambda$ & Term \\
      \midrule
      $((\rank E)^q)$ & weak interpolation \\
      $((\rank E)^q, \lambda_{\textrm{tail}})$ & weak interpolation with tail $\lambda_{\textrm{tail}}$ \\
      $((\rank E)^q, r)$ & (regular) interpolation \\
      all admissible tuples $\lambda$ & strong interpolation \\
      \bottomrule
    \end{tabular}
  \end{center}
\end{definition}

The notion of section stability is related to interpolation.

\begin{definition}[Section stability]
  We call a vector bundle $E$ \emph{section-semistable} if all subbundles $F \subset E$ satisfy the inequality
  \[
  \frac{\h^0(F)}{\rank F} \leq \frac{\h^0(E)}{\rank E}.
  \]
  We will further call $E$ \emph{section-stable} if the strict version of the inequality above holds for all proper $F \subset E$.
\end{definition}

Contrast the notion of section stability with the more classical slope stability.

\begin{definition}[Slope stability]
  A vector bundle $E$ is called \emph{slope-semistable} if all subbundles $F \subset E$ satisfy
  \[
  \frac{c_1(F)}{\rank F} \leq \frac{c_1(E)}{\rank E}.
  \]
  If the strict version of the inequality holds for all proper $F \subset E$, then we call $E$ \emph{slope-stable}.

  The slope of a vector bundle refers to the ratio
  \[
  \mu(E) = \frac{c_1(E)}{\rank E}.
  \]
\end{definition}

\begin{remark}
  As a simple application of Riemann-Roch, we can replace $c_1$ in the definition of slope stability with the Euler characteristic $\chi$ without changing its meaning.
\end{remark}

\begin{remark}
  Let $C$ be a curve embedded in a smooth ambient space $X$. We will use its normal bundle $N_{C/X}$ to extend all properties we defined so far to $C$. For example, we will say that $C$ satisfies interpolation if its normal bundle does so.
\end{remark}

\begin{lemma}
  \label{T:h1-section-slope}
  If $E$ satisfies section (semi)stability and $\h^1(E) = 0$, then $E$ also satisfies slope (semi)stability.
\end{lemma}

\begin{proof}
  The result is immediate from the inequality
  \[
  \frac{\chi(F)}{\rank F} \leq
  \frac{\h^0(F)}{\rank F} \leq
  \frac{\h^0(E)}{\rank E} =
  \frac{\chi(E)}{\rank E}.
  \qedhere
  \]
\end{proof}

Our next goal is to illustrate the definitions of interpolation, slope stability, and section stability using $\Pbb^1$.

\begin{example}[Vector bundles on $\Pbb^1$]
  Vector bundles on the projective line are simple enough that we can characterize their stability and interpolation properties very concretely. Let us fix a rank $n$ vector bundle $E = \bigoplus_{j=1}^n \Oc_{\Pbb^1}(a_j)$.

  Without loss of generality, we assume the $a_j$ are in weakly decreasing order
  \[
  a_1 \geq \cdots \geq a_n.
  \]
  Considering the subbundle $F = \bigoplus_{j < n} \Oc_{\Pbb^1}(a_j) \subset E$, we see that slope semistability implies all $a_j$ are equal. It is easy to show this is a necessary and sufficient condition. Furthermore, the only slope-stable vector bundles have rank $1$.

  Section semistability is equally easy to analyze. Since
  \[
  \h^0(\Oc_{\Pbb^1}(a)) = \max\{ 1 + a, 0 \},
  \]
  the condition we are interested in can be stated as follows: for all subsets $J \subset \{ 1, \dots, n \}$, we have
  \[
  \frac{1}{|J|} \sum_{j \in J} \max\{ a_j + 1, 0 \} \leq
  \frac{1}{n} \sum_j \max\{ a_j + 1, 0 \}.
  \]
  If $E$ has sections and there exists $j'$ such that $a_{j'} < -1$, taking $J = \{ j \neq j' \}$ would violate the inequality. It follows that $a_j \geq -1$ for all $j$, and then we can replace $\h^0(E)$ with the Euler characteristic $\chi(E)$. By Riemann-Roch, this is equivalent to slope semistability. In conclusion, $E = \bigoplus_j \Oc_{\Pbb^1}(a_j)$ is section-semistable if and only if all $a_j$ are no greater than $-1$ or all $a_j$ are equal and no less than $-1$. Equivalently, we can also say that either $E$ has no sections or $E$ is slope-semistable and $\H^1(E) = 0$.
  
  We are ready to analyze interpolation for bundles on $\Pbb^1$. Assume that $\h^1(E) = 0$ which is equivalent to $a_j \geq -1$ for all $j$. We note that a general $(\lambda_1, \dots, \lambda_m)$-modification is the same as a general $(\lambda_2, \dots, \lambda_m)$-modification performed on a general $(\lambda_1)$-modification of $E$.

  Let us consider the $a_j$ in $E = \bigoplus_j \Oc_{\Pbb^1}(a_j)$ as a tableau of height $\rank E$. A general $(\lambda_1)$-modification of $E$ has the effect of removing a box from each of the top $\lambda_1$ rows. Consider the following example with $\lambda_1 = 3$.
  \[
  \young(\ysp\ysp\ysp\x,\ysp\ysp\x,\ysp\x,\ysp\ysp)
  \]
  Note that we may need to rearrange the rows of the result to arrive at a tableau. The number of sections drops by $\lambda_1$ as long as no row in the new tableau has value less than $-1$. The vector bundle $E$ satisfies $(\lambda_1)$-interpolation as long as all $a_j \geq 0$ for $1 \leq j \leq \lambda_1$.

  There is another way to state the condition for $(\lambda_1)$-interpolation. Consider the tableau $(b_i) = (a_j + 1)^T$, that is, we add a column of width $1$ and height $n = \rank E$, and then take the transpose of the result. We illustrate this operation through an example; the new cells are marked with bullet points.
  \begin{align*}
    (a_j) &= \yng(3,3,2,2,2) &
    (b_i) = (a_j + 1)^T &=
    \young(\bullet\bullet\bullet\bullet\bullet,\ysp\ysp\ysp\ysp\ysp,\ysp\ysp\ysp\ysp\ysp,\ysp\ysp)
  \end{align*}
  In these terms $(\lambda_1)$-interpolation is equivalent to the inequality $b_1 \geq \lambda_1$. Extending our logic, it is not too hard to state a condition for $(\lambda_1, \lambda_2)$-interpolation. It turns out we need two inequalities: $\lambda_1 \leq b_1$ and $\lambda_1 + \lambda_2 \leq b_1 + b_2$. Using induction, one can show that $E$ is $\lambda$-stable if and only if
  \[
  \sum_{j' \leq j} \lambda_{j'} \leq
  \sum_{j' \leq j} b_{j'}
  \]
  for all $1 \leq j \leq \rank E$.  

  Our analysis suggests we should introduce the following partial order.
  \begin{definition}[Dominance order]
    Let $\lambda = (\lambda_i)$ and $\lambda' = (\lambda_i')$ be two tableaux. If $\sum_{j \leq i} \lambda_j' \leq \sum_{j \leq i} \lambda_j$ for all $i$, we will say that $\lambda'$ is \emph{dominated} by $\lambda$ and write $\lambda' \trianglelefteq \lambda$.
  \end{definition}

  We will later see that $\lambda$-interpolation implies $\lambda'$-interpolation if $\lambda' \trianglelefteq \lambda$ (see \cref{T:interpolation-implication}). For now, we restate our result for bundles over $\Pbb^1$.
  \begin{proposition}
    \label{T:interpolation-P1}
    Let $E = \bigoplus \Oc_{\Pbb^1}(a_j)$ and $a = (a_j)$. The vector bundle $E$ satisfies $\lambda$-interpolation if and only if $\lambda \trianglelefteq (1+a)^T$.
  \end{proposition}

  We are now ready to analyze weak, regular and strong interpolation. As before, we are assuming $\h^1(E) = 0$. Start by expressing $\h^0(E)$ as
  \[
  \h^0(E) =
  n + \sum_j a_j =
  n \cdot q + r
  \]
  where $0 \leq r < n$. Weak interpolation means that $E$ satisfies $(n^q)$-interpolation, so $a_j + 1 \geq q$ for all $j$. More strongly, regular interpolation is defined as $(n^q, r)$-interpolation which implies
  \[
  a_j =
  \begin{cases}
    q & \textrm{if } 1 \leq j \leq r, \\
    q - 1 & \textrm{if } r+1 \leq j \leq n.
  \end{cases}
  \]
  In other words, $E$ is as balanced as possible. It is not hard to check that strong interpolation does not impose any additional constraints on $E$.

  \begin{remark}
    \label{T:interpolation-distinct}
    Vector bundles on $\Pbb^1$ lead to another interesting observation, namely, that $\lambda$- and $\lambda'$-interpolation are equivalent notions if and only if $\lambda = \lambda'$. It is not difficult to use this strategy and construct similar examples over a curve of any genus.
  \end{remark}

  This completes our account of interpolation properties over $\Pbb^1$. We return to the general case in order to prove several other statements.
\end{example}

\begin{proposition}[Reducing $\lambda$]
  \label{T:reducing-lambda}
  If $\lambda$ and $\lambda'$ are such that $\lambda_i \geq \lambda_i'$ for all $i$, then $\lambda$-interpolation implies $\lambda'$-interpolation.
\end{proposition}

\begin{proof}
  Consider a vector bundle $E$ on $C$ which satisfies $\lambda$-interpolation. Suppose the interpolation condition holds for a specific set of points $p_i \in C$ and subspaces $V_i \subset E_{p_i}$. For each $i$, we pick a superspace $V_i \subset V_i' \subset E_{p_i}$ such that $\dim V_i' = \rank E - \lambda_i'$. The Snake Lemma implies the modifications $F = M(E, V_i)$ and $F' = M(E, V_i')$ fit in the following diagram.
  \[\xymatrix{
    &&& 0 \ar[d] \\
    & 0 \ar[d] && \bigoplus_i V_i'/V_i \ar[d] \\
    0 \ar[r] & F \ar[r] \ar[d] & E \ar[r] \ar@{=}[d] & \bigoplus_i E_{p_i}/V_i \ar[r] \ar[d] & 0 \\
    0 \ar[r] & F' \ar[r] \ar[d] & E \ar[r] & \bigoplus_i E_{p_i}/V_i' \ar[r] \ar[d] & 0 \\
    & \bigoplus_i V_i'/V_i \ar[d] && 0 \\
    & 0
  }\]
  The $\lambda$-interpolation of $E$ implies $\h^1(F) = \h^1(E) = 0$. The long exact sequence of the first column implies $\h^1(F') = 0$. Putting these observations together, we get
  \begin{align*}
    \h^0(F')
    &=
    \h^0(E) - \dim \bigoplus_i E_{p_i}/V_i' =
    \h^0(E) - \sum_i \lambda_i'.
    \qedhere
  \end{align*}
\end{proof}

\begin{proposition}
  \label{T:interpolation-inequality}
  Let $E$ be a vector bundle for which we write $n = \rank E$ and $\h^0(E) = n \cdot q + r$ for $0 \leq r < n$. Suppose $E$ satisfies weak interpolation. If $F \subset E$ is a subbundle, then
  \[
  0 \leq
  \frac{\h^0(F)}{\rank F} \leq
  \frac{\h^0(E)}{\rank E} + r \left( \frac{1}{\rank F} - \frac{1}{\rank E} \right).
  \]
  If $E$ further satisfies interpolation, then
  \[
  0 \leq
  \frac{\h^0(F)}{\rank F} \leq
  \frac{\h^0(E)}{\rank E} + \min\left\{ 1, \frac{r}{\rank F} \right\} - \frac{r}{\rank E}.
  \]
\end{proposition}

\begin{proof}
  The proof of the first statement is analogous to that of the second but uses an easier construction, so we will allow ourselves to only present an argument for the second.

  The hypothesis on $E$ implies that there exists an effective divisor $D$ of degree $q$ such that $\h^0(E(-D)) = r$. Furthermore, there exists a point $p \in C \setminus D$ and a subspace $V \subset E_p$ of codimension $r$ such that the modification $E' = M(E(-D), V)$ has no sections. We construct the modification $F' = M(F(-D), W)$ for $W = V \cap F_p$. A basic fact about dimension of intersections implies $\max\{ 0, \rank F - r \} \leq \dim W$, so
  \[
  \dim(F_p/W) \leq \min\{ \rank F, r \}.
  \]
  
  The constructions above ensure that there is a injection $F' \rarr E'$, hence $\h^0(F') \leq \h^0(E') = 0$. Furthermore, the short exact sequence
  \[\xymatrix{
    0 \ar[r] &
    F' \ar[r] &
    F \ar[r] &
    F|_D \oplus F_p/W \ar[r] &
    0
  }\]
  implies
  \[
  \h^0(F) \leq
  \h^0(F') + \h^0( F|_D \oplus F_p/W ) =
  q \cdot \rank F + \min\{ \rank F, r \}.
  \]
  Dividing by $\rank F$, we arrive at the desired inequality.
\end{proof}

Setting $r = 0$ in the proof of \cref{T:interpolation-inequality}, we obtain the following result.

\begin{corollary}
  If $E$ satisfies weak interpolation and its rank divides $\h^0(E)$, then $E$ is section-semistable.
\end{corollary}

\begin{lemma}
  \label{T:global-sections-subbundle}
    If $E$ is a vector bundle whose global sections span a subbundle of rank $r$, then $E$ satisfies $(r)$-interpolation.
\end{lemma}

\begin{proof}
  Let $F \subset E$ be the subbundle spanned by global sections (the saturation of the subsheaf generated by global sections). If $p$ is a general point on the underlying curve, then the evaluation morphism $F \rarr F_p$ is surjective on global sections. By picking a subspace $V \subset E_p$ complimentary to $F_p$ we arrive at the following diagram with two exact rows.
  \[\xymatrix{
    0 \ar[r] &
    F(-p) \ar[r] \ar[d] &
    F \ar[r] \ar[d] &
    F_p \ar[r] \ar[d]^-{\cong} &
    0 \\
    0 \ar[r] &
    M(E,V) \ar[r] &
    E \ar[r] &
    E_p/V \ar[r] &
    0
  }\]
  By construction $F \rarr E$ is an isomorphism on global sections, so $E \rarr E_p/V$ must be surjective on global sections which proves the required claim.
\end{proof}

\begin{lemma}
  \label{T:adding-1}
  If $E$ satisfies $\lambda$-interpolation, then $E$ also satisfies $(\lambda, 1^k)$-interpolation for $k = \h^0(E) - \sum_i \lambda_i$.
\end{lemma}

\begin{proof}
  If $F$ denotes a general $\lambda$-modification of $E$, then it suffices to show $F$ satisfies $(1^{\h^0(F)})$-interpolation. Furthermore, it suffices to show that if $h^0(F) \geq 1$, then $F$ satisfies $(1)$-interpolation. By considering a non-zero section of $F$, it follows that the subbundle spanned by global sections has rank at least $1$ and we can apply \cref{T:global-sections-subbundle}.
\end{proof}

\begin{corollary}
  \label{T:strong-interpolation-2}
  For rank $2$ bundles, weak, regular and strong interpolation are equivalent notions.
\end{corollary}

\begin{proof}
  Weak interpolation means that $E$ satisfies $(2^q)$-interpolation where $q = \lfloor \h^0(E)/2 \rfloor$. To demonstrate strong interpolation, we need to consider all tableaux $(2^a, 1^b)$ such that $2a + b \leq \h^0(E)$. Since $a \leq q$, \cref{T:reducing-lambda} implies that $E$ satisfies $(2^a)$-interpolation. To arrive at $(2^a, 1^b)$ we apply \cref{T:adding-1}.
\end{proof}

In fact, the equivalence of strong and regular interpolation extends to arbitrary rank vector bundles; we postpone the discussion of this result to \cref{S:strong-regular}.

The following two diagrams summarize the implications we described so far.

\vspace{\baselineskip}
\begin{centering}
  \begin{tabular}{p{3in}p{3in}}
    \textbf{In genus $0$} &
    \textbf{In general} \\[10pt]
    $\xymatrix@C=0.07in{
      \textrm{strong interpolation} \ar@{<=>}[d] &
      \textrm{slope semistability} \ar@{=>}@/^/[d]^-{\h^1 = 0} \\
      \textrm{interpolation} \ar@{=>}[d] \ar@{=>}[r] &
      \textrm{section semistability} \ar@{=>}@/^/[u]^-{\h^0 > 0} \\
      \textrm{weak interpolation}
    }$
    &
    $\xymatrix@C=0.07in{
      \textrm{strong interpolation} \ar@{<=>}[d]_-{\cref{T:strong-interpolation}} &
      \textrm{slope semistability} \ar@{<=}[d]^-{\h^1 = 0} \\
      \textrm{interpolation} \ar@/^/@{=>}[d] &
      \textrm{section semistability} \\
      \textrm{weak interpolation} \ar@{=>}[ur]_-{\rank | \h^0} \ar@{=>}@/^/[u]^-{\rank = 2}
    }$
  \end{tabular}
\end{centering}


\section{Interpolation via nodal curves}
\label{S:nodal-curves}

We aim to prove the following two results.

\begin{theorem}
  \label{T:two-curves-interpolation}
  Let $C$ and $D$ be curves in some ambient projective space $\Pbb^n$ which meet transversely in a single point $p$, smooth in each curve. Assume
  \begin{enumeratea}
  \item $C$ satisfies $\lambda$-interpolation,
  \item $D$ satisfies $(n-1, \mu)$-interpolation.
  \end{enumeratea}
  Then, up to changing the point of attachment on $D$ through a rigid motion (via the action of $\PGL(n+1)$ on $\Pbb^n$), the union curve $X = C \cup D$ satisfies $(\lambda, \mu)$-interpolation.
\end{theorem}

\begin{theorem}
  \label{T:curve-secant-interpolation}
  Let $C \subset \Pbb^n$ be a curve, and $L$ a secant line which meets $C$ transversely in two points $p$, $q$, both smooth in $C$. If $C$ satisfies $\lambda$-interpolation, then these properties also hold for the union $X = C \cup L$.

  If the two tangent lines to $C$ at the points of $C \cap L$ are skew, then $X$ satisfies $(\lambda, 2)$-interpolation. This condition holds for a general secant line as long as $C$ is not planar.
\end{theorem}

The proofs use techniques similar to \cite{Hartshorne-Hirschowitz} and \cite{Ran}, so we start by recalling several preliminary results these sources present.

\begin{proposition}
  \label{T:normal-bundle-normalization}
  Let $X \subset \Pbb^n$ be a nodal curve, and let $\nu \cn X' \rarr X$ denote its normalization. Then the normal bundles of the two curves fit in a short exact sequence
  \[\xymatrix{
    0 \ar[r] &
    N_{X'} \ar[r] &
    \nu^\ast N_X \ar[r] &
    \nu^\ast T^1_X \ar[r] &
    0,
  }\]
  where $T^1_X$ stands for the Lichtenbaum-Schlessinger $T^1$-functor of $X$. It is supported on $S = \Sing X$, so we will also denote it by $T^1_S$.
\end{proposition}

\begin{remark}
  \label{T:nodal-curve-normal-bundle}
  Let $i \cn X \rarr \Pbb^n$ denote the inclusion morphism as in the previous result. While the composition $i \circ \nu \cn X' \rarr \Pbb^n$ is not an embedding, it is an embedding locally on $X'$. It follows that the differential $d(i \circ \nu) \cn T_{X'} \rarr (i \circ \nu)^\ast T_{\Pbb^n}$ an injective morphism of vector bundles, so we can define the normal bundle $N_{X'}$ as its cokernel.
\end{remark}

\begin{corollary}
  \label{T:normal-bundle-component}
  Let $X = C \cup D$ be the union of two smooth curves $C$, $D$ intersecting transversely in $S$. Then there is a short exact sequence
  \[\xymatrix{
    0 \ar[r] &
    N_C \ar[r] &
    N_X|_C \ar[r] &
    T^1_S \ar[r] &
    0.
  }\]
\end{corollary}

We are now ready to substantiate our claims.

\begin{proof}[Proof of \cref{T:two-curves-interpolation}]
  First pick a $\lambda$-modification datum $V_i \subset N_{C, p_i}$ away from $C \cap D = \{p\}$, so $\h^0(M(N_C, V_i)) = \h^0(N_C) - \sum_i \lambda_i$. We do the same for $D$ but only using a $\mu$-modification: $W_j \subset N_{D, q_j}$ is a datum away from $p$ such that $\h^0(M(N_D, W_j)) = \h^0(N_D) - \sum_j \mu_j$ and $M(N_D, W_j)$ satisfies $(n-1)$-interpolation.
  
  After applying a modification to the sequence from \cref{T:normal-bundle-component}, we obtain the following diagram with two exact rows and two exact columns.
  \[\xymatrix{
    & 0 \ar[d] & 0 \ar[d] \\
    0 \ar[r] & M(N_C, V_i) \ar[r] \ar[d] & M(N_X|_C, V_i) \ar[r] \ar[d] & T^1_S \ar[r] \ar@{=}[d] & 0 \\
    0 \ar[r] & N_C \ar[r] \ar[d] & N_X|_C \ar[r] \ar[d] & T^1_S \ar[r] & 0 \\
    & \bigoplus_i N_{C, p_i}/V_i \ar[d] \ar@{=}[r] & \bigoplus_i N_{C, p_i}/V_i \ar[d] \\
    & 0 & 0
  }\]
  A diagram chase argument shows that all sheaves present have $\H^1 = 0$, so the diagram remains exact after applying $\H^0$ to it. The exactness of the second column implies
  \[
  \h^0(M(N_X|_C, V_i)) = \h^0(N_X|_C) - \sum \lambda_i.
  \]
  The analogous argument for $D$ produces
  \[
  \h^0(M(N_X|_D, W_j)) =
  \h^0(N_X|_D) - \sum \mu_j.
  \]
  Next, we need to relate $N_X|_C$ and $N_X|_D$ to $N_X$. Tensor the short exact sequence
  \[\xymatrix{
    0 \ar[r] &
    \Oc_X \ar[r] &
    \Oc_C \oplus \Oc_D \ar[r] &
    \Oc_p \ar[r] &
    0,
  }\]
  with the morphism $M(N_X, V_i, W_j) \rarr N_X$.
  \[\xymatrix{
    0 \ar[r] &
    M(N_X, V_i, W_j) \ar[r] \ar[d] &
    M(N_X|_C, V_i) \oplus M(N_X|_D, W_j) \ar[r] \ar[d] &
    M(N_X, V_i, W_j)|_p \ar[r] \ar[d]^-{\cong} &
    0 \\
    0 \ar[r] &
    N_X \ar[r] &
    N_X|_C \oplus N_X|_D \ar[r] &
    N_X|_p \ar[r] &
    0
  }\]
  If we assume $\H^1(M(N_X, V_i, W_j)) = 0$, then $\H^1(N_X) = 0$ and we deduce
  \begin{align*}
    \h^0(M(N_X, V_i, W_j))
    &= \h^0(M(N_X|_C, V_i)) + \h^0(M(N_X|_D, W_j)) - \h^0(M(N_X, V_i, W_j)|_p) \\
    &= \h^0(N_X|_C) - \sum \lambda_i + \h^0(N_X|_D) - \sum \mu_j - \h^0(N_X|_p) \\
    &= \h^0(N_X) - \sum \lambda_i - \sum \mu_j,
  \end{align*}
  which proves that $X$ satisfies $(\lambda,\mu)$-interpolation.

  We are left to show that $\H^1(M(N_X, V_i, W_j)) = 0$. From the previous diagram, it suffices to show that
  \[\xymatrix{
    M(N_X|_D, W_j) \ar[r] &
    M(N_X, V_i, W_j)|_p = M(N_X, W_j)|_p
  }\]
  is surjective on $\H^0$. This follows if we show that
  \[\xymatrix{
    M(N_D, W_j) \ar[r] &
    M(N_D, W_j)|_p
  }\]
  is surjective on $\H^0$, but that is a consequence of the fact $M(N_D, W_j)$ satisfies $(n-1)$-interpolation. A rigid motion of one of the curves relative to the other may be necessary since the point of modification $p$ for $D$ has to be general.
\end{proof}

Taking $D$ to be a line leads to a very useful result.

\begin{corollary}
  \label{T:curve-line-interpolation}
  Let $C \subset \Pbb^n$ be a curve, and $L$ a line which meets $C$ transversely in a single point $p$ smooth on $C$. If $C$ satisfies $\lambda$-interpolation, then $X = C \cup L$ satisfies $(n-1, \lambda)$-interpolation.
\end{corollary}

\begin{proof}
  The normal bundle to $L$ is $N_L \cong \Oc_L(1)^{\oplus(n-1)}$, so it satisfies $((n-1)^2)$-interpolation. No reattachment is necessary since two points on a line have no moduli.
\end{proof}

We need several ``selection'' results to complete the proof of \cref{T:curve-secant-interpolation}. These statements allow us to choose objects (e.g., secant lines, tangent lines) with sufficiently nice properties.

\begin{proposition}
  \label{T:tangent-away-from-line}
  Let $C \subset \Pbb^n$ be a reduced irreducible curve and $L \subset \Pbb^n$ be a line. If $C$ is non-planar, then a general tangent line of $C$ does not meet $L$.
\end{proposition}

\begin{proof}
  For contradiction, assume all tangent lines of $C$ meet $L$. There exists $p \in L \setminus C$, otherwise $L \subset C$ and $C$ is planar. We consider the projection $\pi \cn \Pbb^n \setminus \{p\} \rarr \Pbb^{n-1}$ away from $p$. Let $\ell \in \Pbb^{n-1}$ be the image of $L$ and $C'$ the image of $C$. The incidence condition for $C$ and $L$ implies that all tangent lines of $C'$ pass through the point $\ell$. Hartshorne calls such curves \emph{strange} and proves they are either a line (in any characteristic) or a conic (in characteristic 2) \cite[Theorem IV.3.8]{Hartshorne} (the original result is attributed to \cite{Samuel}). We are excluding the latter case, so $C'$ must be a line. Then $C$ lies in the preimage of this line which is a plane, contradicting our hypotheses.
\end{proof}

\begin{corollary}
  \label{T:skew-tangents}
  Let $C \subset \Pbb^n$ be a reduced irreducible curve. If $C$ is non-planar, then any finite set of general tangent lines is pairwise disjoint.
\end{corollary}

\begin{proof}
  We induct on the number of tangent lines we need to construct. In the inductive step we invoke \cref{T:tangent-away-from-line} for each of the previously chosen lines.
\end{proof}

\begin{proposition}
  \label{T:secant-away-from-line}
  Let $C \subset \Pbb^n$ be a reduced irreducible curve and $L \subset \Pbb^n$ a line. If $C$ is non-planar, then a general secant line does not meet $L$.
\end{proposition}

\begin{proof}
  Since $C$ is non-planar, there exists a point $p \in C \setminus L$. If all secants through $p$ meet $L$, then $C$ is contained in the plane spanned by $L$ and $p$ which is a contradiction.
\end{proof}

\begin{corollary}
  \label{T:skew-secants}
  Let $C \subset \Pbb^n$ be a reduced irreducible curve. If $C$ is non-planar, then any finite set of general secant lines is pairwise disjoint.
\end{corollary}

\begin{proof}
  The argument is analogous to \cref{T:skew-tangents}.
\end{proof}

\begin{proposition}[Trisecant lemma]
  \label{T:trisecant-lemma}
  Let $C$ be a smooth curve. If $C$ is non-planar, then
  \begin{enumeratea}
  \item
    a general secant line is not trisecant, and
  \item
    a general tangent line intersects $C$ in a single point only.
  \end{enumeratea}

\end{proposition}

\begin{proof}
  This is a combination of \cite[Proposition IV.3.8]{Hartshorne} and \cite[Theorem IV.3.9]{Hartshorne}.
\end{proof}

The next result can be stated more easily using the following notion.

\begin{definition}
  Let $C \subset \Pbb^n$ be a smooth curve and $E$ a subbundle of the restriction of the ambient tangent bundle $T_{\Pbb^n}|_C$. Consider a second curve $D$ which shares no components with $C$ and the intersection $C \cap D$ lies in the smooth locus of $D$.
  
  We will say that $D$ is \emph{away} from $E$ at a point $p \in C \cap D$ if the tangent line $T_p D$ is not contained in the fiber $E_p$. We will say that $D$ is \emph{away} from $E$ if the previous condition holds at all points of intersection $p \in C \cap D$.

  If instead $E$ is a subbundle of the normal bundle $N_C$, we make the same definitions by replacing it with its preimage in $T_{\Pbb^n}|_C$.
\end{definition}

\begin{proposition}
  \label{T:general-secant-away}
  Let $C \subset \Pbb^n$ be a smooth irreducible curve and $E \subset T_{\Pbb^n}|_C$ a proper subbundle, that is, its rank is strictly less than $n$. If $C$ spans the ambient space $\Pbb^n$, then a general secant line is away from $E$.
\end{proposition}

\begin{proof}
  Let $\pr_1 \cn C \x C \rarr C$ denote the projection to the first factor. We construct a line bundle $M \subset \pr_1^\ast T_{\Pbb^n}|_C$ as follows. Consider a point $(p,q) \in C \x C$. If $p \neq q$, let $L_{p,q}$ denote the secant line joining $p$ and $q$, and if $p = q$, then $L_{p,p}$ will stand for the tangent line of $C$ at $p$. The bundle $M$ is defined by saying its fiber over $(p,q)$ is $T_{L_{p,q},p} \subset T_{\Pbb^n,p}$. This construction can be made precise by assembling the lines $L_{p,q}$ into a $\Pbb^1$-bundle over $C \x C$.

  Note that $L_{p,q}$ is away from $E$ at $p$ if and only if $M_{p,q}$ is not contained in $E_p$. Let $U \subset C \x C$ denote the open locus where $M$ is transverse to $\pr_1^\ast E$. Since the problem is symmetric in the two points on the secant, it is easy to reverse it. Let $\sigma \cn C \x C \rarr C \x C$ denotes the swap morphism. A secant line $L_{p,q}$ is away from $E$ if and only if $(p,q)$ lies in $U \cap \sigma(U)$.

  It suffices to show that $U$ is not empty to infer that neither is $U \cap \sigma(U)$ which completes our argument. For contradiction, assume that $M \subset \pr_1^\ast E$. Fix a point $p \in C$ and let $H$ be the unique linear space through $p$ of dimension $\rank E$ such that $T_{H, p} = E_p$. Consider another point $q \in C$ and the secant (tangent) line $L_{p,q}$. Since $M \subset \pr_1^\ast E$ and $M_{p,q} = T_{L_{p,q},p}$, it follows that $L_{p,q}$ is contained in $H$. A fortiori, the point $q$ lies in $H$, hence so does the entire curve $C$. We are given that $C$ spans the ambient $\Pbb^n$, so $H = \Pbb^n$ contradicting the properness of $E$.
\end{proof}

We are now ready to prove the second main result of this section.

\begin{proof}[Proof of \cref{T:curve-secant-interpolation}]
  We proceed analogously to the Proof of \cref{T:two-curves-interpolation}. Pick a modification datum $V_i \subset N_{C, p_i}$ away from $S = C \cap L$, so $\h^0(M(N_C, V_i)) = \h^0(N_C) - \sum_i \lambda_i$ and $\h^1(M(N_C, V_i)) = 0$. Since $N_X$ and $N_C$ are identical on $C \setminus S$, this gives us a modification datum for $N_X$ and we aim to show the analogous statement $h^0(M(N_X, V_i)) = h^0(N_X) - \sum_i \lambda_i$.

  First, we deduce
  \[
  \h^0(M(N_X|_C, V_i)) = \h^0(N_X|_C) - \sum \lambda_i.
  \]
  If we assume $\H^1(M(N_X, V_i)) = 0$, then
  \[
  \h^0(M(N_X, V_i)) = \h^0(N_X) - \sum \lambda_i,
  \]
  which is the desired interpolation property for $N_X$. It suffices to show the morphism $\H^0(N_X|_L) \rarr \H^0(N_X|_S)$ is surjective. But $N_L \cong \Oc_L(1)^{\oplus(n-1)}$ and $N_X|_L \cong \Oc_L(1)^{\oplus(n-3)} \oplus \Oc_L(2)^{\oplus 2}$ or $N_X|_L \cong \Oc_L(1)^{\oplus(n-2)} \oplus \Oc_L(3)$. In each of the cases, the degrees of all components are bounded from below by $1$ and the statement holds.

  For the second claim we need to prove, consider a secant line $L$ to $C$ whose two tangent lines are skew. \cref{T:skew-tangents} shows this condition is satisfied for a general secant line. If we assume $C$ satisfies $\lambda$-interpolation, our goal is to show that $X = C \cup L$ satisfies $(\lambda,2)$-interpolation via the addition of a point and codimension $2$ space on $L$. It suffices to show that $N_X|_L \cong \Oc_L(1)^{\oplus (n-3)} \oplus \Oc_L(2)^{\oplus 2}$ and not $N_X|_L \cong \Oc_L(1)^{\oplus(n-2)} \oplus \Oc_L(3)$. The normal bundle $N_X|_L$ remains unchanged if we replace $C$ with the union $L_1 \cup L_2$ of the two tangent lines to $C$ at $C \cap L$. Let $P_1$ denote the 2-plane spanned of $L$ and $L_1$, $P_2$ the 2-plane spanned by $L$ and $L_2$, and $P$ the 3-plane spanned by $L$, $L_1$ and $L_2$. We can use these to show that
  \[
  N_X|_L \cong N_{(L \cup L_1)/P_1}|_L \oplus N_{(L \cup L_2)/P_2}|_L \oplus N_{P/\Pbb^n}|_L.
  \]
  By studying the union of two distinct lines in the plane, we arrive at the conclusion that the three summands are isomorphic to $\Oc_L(2)$, $\Oc_L(2)$ and $\Oc_L(1)^{\oplus(n-3)}$ respectively. This completes our claim.
\end{proof}


\section{Remarks on Hilbert schemes}
\label{S:Hilbert-schemes}

The results from \cref{S:nodal-curves} allow us to prove interpolation for some special reducible nodal curves. Building on knowledge about simpler curves, we can use these ideas to deduce interpolation properties for a wider class of reducible nodal curves. On the other hand, we would like to prove statements about smooth curves. In this section we demonstrate how to transfer information using a family of curves, hence opening the opportunity for the degeneration arguments customary in algebraic geometry. Hilbert schemes play a central role in this discussion.

For any triple $d$, $g$ and $n$, we can construct the Hilbert scheme $\Hc_{d,g,n}$ whose points parametrize curves of degree $d$ and genus $g$ in $\Pbb^n$. We consider the locus of smooth irreducible curves $\Ic_{d,g,n}' \subset \Hc_{d,g,n}$, and its closure $\Ic_{d,g,n}$ also referred to as the \emph{restricted Hilbert scheme}.

Suppose we have a nodal curve $C$ such that $[C] \in \Hc_{d,g,n}$. In order to use $C$ to reach any statements about smooth curves, we need to verify it lies in $\Ic_{d,g,n}$. Hartshorne and Hirschowitz consider this problem in \cite{Hartshorne-Hirschowitz}. We recall several useful statements from their work.

\begin{definition}
  \label{T:strongly-smoothable}
  We call a curve $C \subset \Pbb^n$ \emph{strongly smoothable} if there exists a flat family $\Cc \rarr \Pbb^n \x B$ with special fiber $C$ such that
  \begin{enumeratea}
  \item
    the general member $\Cc_t$ is smooth,
  \item
    the total space $\Cc$ is smooth and irreducible, and
  \item
    the base $B$ is smooth.
  \end{enumeratea}
\end{definition}

\begin{remark}
  Being strongly smoothable is stronger than containment in $\Ic_{d,g,n}$, also called (regular) smoothability.
\end{remark}

\begin{proposition}
  \label{T:curve-secant-smoothable}
  Let $C$ be a smooth curve in $\Pbb^n$ with $\H^1(N_C) = 0$. If $L$ is a line meeting $C$ transversely in one or two points, then $\H^1(N_X) = 0$ and $X = C \cup L$ is strongly smoothable.
\end{proposition}

\begin{proposition}
    \label{T:curve-lines-smoothable}
  Let $X = C \cup L_1 \cup \cdots \cup L_a \subset \Pbb^n$ be the union of a smooth irreducible curve $C$ with lines, each meeting $C$ transversely in a single point. If $a \geq \h^1(N_C) + 1$, then $X$ is strongly smoothable. If in addition $\H^1(N_C) = 0$ or the lines $L_i$ are general, then $\H^1(N_X) = 0$.
\end{proposition}

By studying the proofs of these statements, there is a simple way to combine some aspects from each.

\begin{proposition}
  \label{T:curve-secant-lines-smoothable}
  Let $C \subset \Pbb^n$ be a smooth irreducible curve, and $L_1, \dots, L_a, L_1', \dots, L_b'$ be lines such that
  \begin{enumeratea}
  \item
    each $L_i$ meets $C$ in a single point,
  \item
    each $L_j'$ is a secant to $C$ (it meets $C$ in exactly two points),
  \item
    all lines $L_i$, $L_j'$ are disjoint, and
  \item
    any line intersects $C$ transversely.
  \end{enumeratea}
  Consider the union curve $X = C \cup L_1 \cup \cdots \cup L_a \cup L_1' \cup \cdots \cup L_b'$. If $\H^1(N_C) = 0$, then $\H^1(N_X) = 0$ and $X$ is strongly smoothable.
\end{proposition}

It is often convenient to refer to a general (smooth) curve. Since we are discussing the Hilbert scheme of curves, this is an appropriate place to elaborate on the meaning of ``general''. If $\Ic_{d,g,n}$ is irreducible, then referring to a general curve with the given characteristics in unambiguous. If however that is not the case, we might want to make separate claims for each component. In such cases, we avoid using the adjective ``general'' and attempt to be more specific. The good news is most cases of interest fall in the former category.

Combining the irreducibility of $\Mc_g$ (see \cite{Deligne-Mumford}) with Riemann-Roch yields an irreducibility statement for $\Ic_{d,g,n}$. For further discussion of this and other related results, we refer to \cite{Harris-Montreal-notes}.

\begin{theorem}
  \label{T:Hilbert-scheme-irreducibility}
  If $d \geq 2g-1$ and $d \geq n + g$, the space $\Ic_{d,g,n}$ is irreducible.
\end{theorem}

Note that when $g = 1$, the hypotheses of \cref{T:Hilbert-scheme-irreducibility} are always satisfied for an elliptic curve which spans the ambient space.

Space curves are another interesting source of examples. There has been a sequence of results extending the range from \cref{T:Hilbert-scheme-irreducibility}. The strongest version we are aware of is due to S.J.~Keem and C.~Kim.

\begin{theorem}[Kim-Keem \cite{Keem-Kim, Keem}]
  \label{T:Hilbert-scheme-irreducibility-P3}
  The space $\Ic_{d,g,3}$ is irreducible in the following cases:
  \begin{enumerate}[(a)]
  \item
    $d \geq g + 3$ and $g \geq 0$,
  \item
    $d = g + 2$ and $g \geq 5$,
  \item
    $d = g + 1$ and $g \geq 9$.
  \end{enumerate}
\end{theorem}

\begin{figure}[h!]
  \centering
  \includegraphics[width=0.75\textwidth]{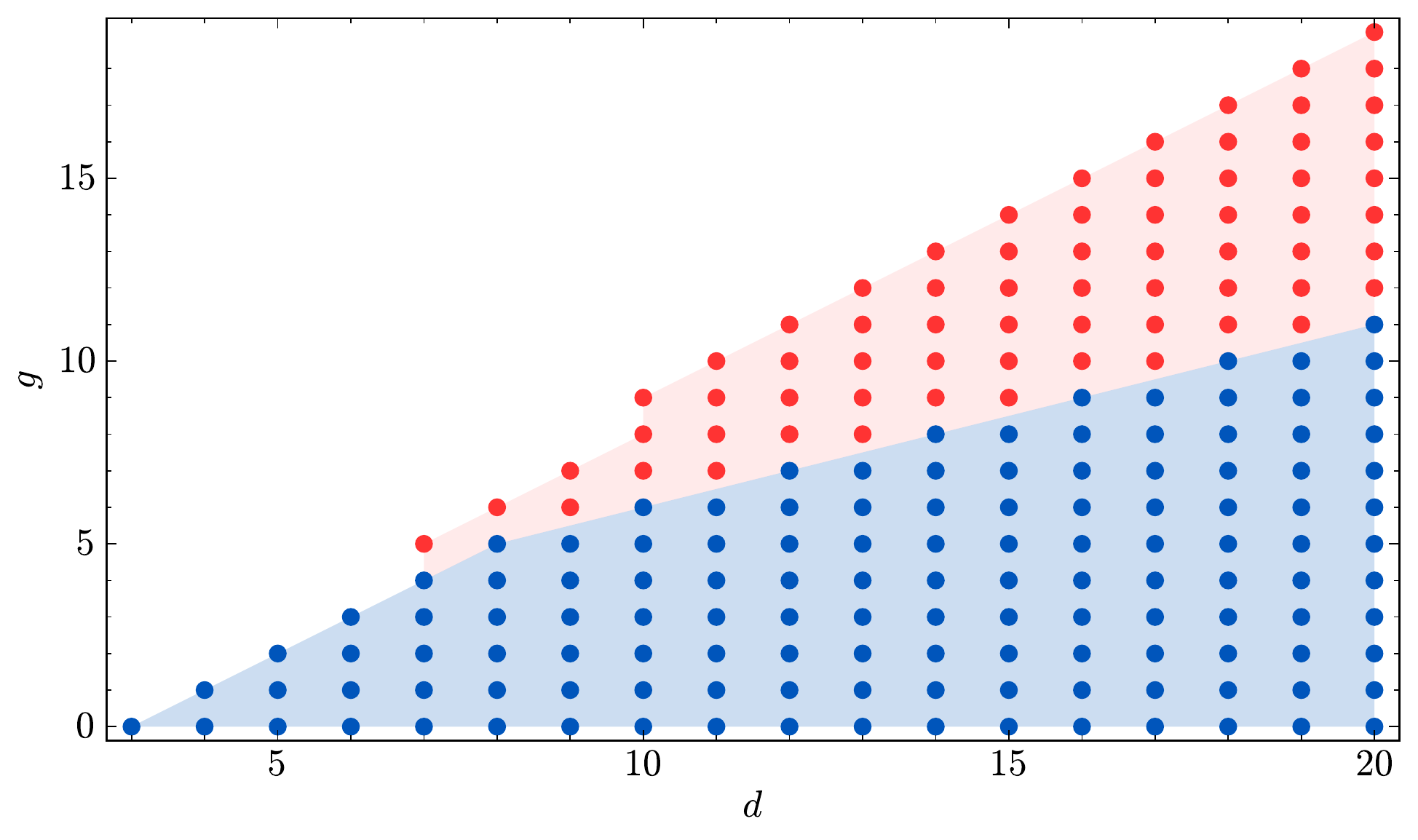}
  \caption{The shaded region of the $(d,g)$-plane illustrates the cases for which we know $\Ic_{d,g,3}$ is irreducible. \cref{T:Hilbert-scheme-irreducibility} corresponds to the blue region (darker in grayscale). The additional part covered by \cref{T:Hilbert-scheme-irreducibility-P3} is in red (lighter in grayscale).}
  \label{F:Hilbert-P3}
\end{figure}
  
The most useful result of this section says that interpolation is an open property. We can then use degeneration arguments to demonstrate interpolation.

\begin{theorem}
  \label{T:interpolation-open}
  Let $\pi \cn \Cc \rarr B$ be a flat family of curves and $E$ a vector bundle over $\Cc$. For each point $b \in B$, let $\Cc_b = \pi^{-1}(b)$ denote the curve over $b$ and $E_b = E|_{\Cc_b}$ be the associated vector bundle. Assume that both $B$ and $\Cc$ are irreducible, and in addition, $B$ is reduced. If there exists a point $b_0 \in B$ such that
  \begin{enumeratea}
  \item
    the curve $\Cc_{b_0}$ is reduced,
  \item
    the value $\H^0(E_{b_0})$ is minimal among $\H^0(E_b)$ for $b \in B$, and
  \item
    vector bundle $E_{b_0}$ satisfies $\lambda$-interpolation,
  \end{enumeratea}
  then $E_b$ satisfies $\lambda$-interpolation for a general point $b \in B$.
\end{theorem}

\begin{proof}
  We start by disposing of the non-reduced curves in the family. Since $\Cc_{b_0}$ is reduced and $\pi \cn \Cc \rarr B$ is flat, then the locus of reduced curves $B' \subset B$ is a non-empty open. In what follows, we replace the morphism $\pi$ with its restriction $\pi|_{\pi^{-1}(B)} \cn \pi^{-1}(B') \rarr B'$.

  The key ingredient in this argument is upper semi-continuity of cohomology. To present a more readable account, we will first assume $\lambda = (\lambda_1)$ and later explain how the general case is handled.

  The strategy is to construct a space whose points parametrize possible modifications. Let $U \subset \Cc$ denote the smooth locus of $\pi$. In other words, $U$ consists of points $p$ such that $p$ is a smooth point of the fiber $\Cc_{\pi(p)}$ it belongs to. The space $U$ models all points at which a modification can be performed. Given $p \in U$, the datum of a $(\lambda_1)$-modification also requires a subspace $V_1 \subset E_p$ of codimension $\lambda_1$. The configurations $(p, V_1)$ are parametrized by the total space of the Grassmannian bundle $\gamma \cn G = \Gr(\rank E - \lambda_1, E|_U) \rarr U$. Let $\gamma^\ast E \rarr Q$ denote the universal quotient over $G$. When there is no ambiguity, we will use $E$ to refer to the restriction $E|_U$.

  Since we are not assuming the total space of the family $\Cc$ is irreducible, it is possible that the smooth locus $U$ of $\pi \cn \Cc \rarr B$ is reducible and disconnected (see \cref{T:smooth-components}). By hypothesis, there exists a modification datum over $b_0$ such that $E_{b_0}$ satisfies $\lambda$-interpolation. Let $U'$ denote the (connected) component of $U$ containing the base of this datum and set $G' = \gamma^{-1}(G')$. Since $\gamma|_{G'} \cn G' \rarr U'$ is a Grassmannian bundle and $U'$ is connected and irreducible, then so is $G'$. For the sake of simplicity, we replace $\gamma \cn U \rarr G$ with $\gamma|_{G'} \cn G' \rarr U'$ for the rest of the argument. This will allow us to refer to general points without any ambiguity.
  
  We will use the following diagram of morphisms in the rest of the argument: $p_0$ and $p_1$ denote projections on the first and second factors respectively, $q$ is $\id_C$ in the first factor and the Grassmannian projection $\gamma \cn G \rarr U$ in the second, $q_i = p_i \circ q$ for $i = 0,1$, and $r$ is projection onto the second factor.
  \[\xymatrix{
    & \Cc \x_B G \ar[d]^-{q} \ar@/_1.5pc/[ddl]_-{q_0} \ar@/^1.5pc/[ddr]^-{q_1} \ar[rr]^-{r} && G \ar@/^/[ddl]^-{\gamma} \\
    & \Cc \x_B U \ar[dl]^-{p_0} \ar[dr]_-{p_1} \\
    \Cc && U
  }\]
  
  The diagonal $\Delta \subset U \x_B U$ sits naturally in $\Cc \x_B U$ where it is a Cartier divisor. Note that the diagonal of $\Cc \x_B \Cc$ is not Cartier, and this is the precise reason we need to restrict $\Cc$ to $U$ on the right. There is a natural isomorphism $(p_0^\ast E)|_\Delta \cong (p_1^\ast E)|_\Delta$. The composition
  \[\xymatrix{
    q_0^\ast E =
    q^\ast p_0^\ast E \ar[r] &
    q^\ast (p_0^\ast E)|_{\Delta} \cong
    q^\ast (p_1^\ast E)|_{\Delta} =
    (q_1^\ast E)|_{q^{-1}(\Delta)} =
    (r^\ast \gamma^\ast E)|_{q^{-1}(\Delta)} \ar[r] &
    (r^\ast Q)|_{q^{-1}(\Delta)}
  }\]
  is surjective and we will use $E'$ to denote its kernel. We generalized the construction of modifications to a family, so the expectation is that for every $g \in G$ the restriction $E'|_{r^{-1}(g)}$ is isomorphic to the modification $M(E_{\pi \circ \gamma(g)}, g)$ over $\Cc_{\pi \circ \gamma(g)}$. We ensured this the case since $\gamma(g) \in U \subset \Cc$.

  Since $\Delta \subset \Cc \x_B U$ is Cartier, it follows that $E'$ is locally free, and in particular, $E'$ is flat over $G$. Cohomology and base change implies that the function $\phi \cn G \rarr \Zbb_{\geq 0}$ sending $g$ to
  \[
  \h^0(r_\ast E'|_g) =
  \h^0(E'|_{r^{-1}(g)}) =
  \h^0(M(E_{\pi \circ \gamma(g)}, g))
  \]
  is upper semi-continuous. In a similar note, sending $b \in B$ to $\h^0(E_b)$ is also upper semi-continuous. Since $B$ is irreducible, $\h^0(E_b)$ attains a minimum value $c$ over a dense open $V \subset B$. The short exact sequence defining a modification implies that $\phi$ is bounded from below by $c - \lambda_1$. We already know there is a point $g \in (\pi \circ \gamma)^{-1}(b_0)$ which satisfies $\phi(g) = c - \lambda_1$, so this minimum is attained over a dense open $G'$ in $G$. The original claim amounts to saying the image $\pi \circ \gamma(G')$ in $B$ is open. This follows from Chevalley's Theorem on constructible sets (see \cite[p.~94]{Hartshorne})
  
  The argument above demonstrates the desired claim when $\lambda = (\lambda_1)$. In the general case $\lambda = (\lambda_1, \dots, \lambda_m)$ we replace the second factor $U$ with the complement $V$ of the big diagonal in the $m$-fold product $U \x_B \cdots \x_B U$. Likewise, the moduli space of all modifications will be a product of $m$ Grassmannian bundles
  \[
  G = \prod_{i = 1}^m \Gr(\rank E - \lambda_i)
  \]
  and $V \subset U \x_B \cdots \x_B U$ will parametrize points $(p_1, \dots, p_m)$ such that all $p_i$ are distinct and they are smooth points of the corresponding fiber. Other than these changes, the argument goes through without any major modifications.
\end{proof}

\begin{corollary}
  \label{T:interpolation-open-Hilbert}
  Let $\Ic$ be an irreducible component of some reduced Hilbert scheme $\Ic_{d,g,n}$. Consider a curve $C$ in $\Ic$ such that $\H^1(N_C) = 0$. If $C$ satisfies $\lambda$-interpolation, then the general curve in $\Ic$ also satisfies $\lambda$-interpolation.
\end{corollary}

\begin{proof}
  Let $\pi \cn \Cc \rarr \Ic$ be the universal curve over $\Ic$. Since $\Ic$ is irreducible and the general member fiber of $\pi$ is smooth, it follows that $\Cc$ is irreducible. Since $[C] \in \Hc_{d,g,n}$ is a smooth point parametrizing a smoothable curve, we also know that $B$ is reduced and $\Cc$ is irreducible. In addition, $\H^1(N_C) = 0$ so the value $\H^0(N_C)$ is minimal. The desired claim follows directly by applying \cref{T:interpolation-open}.
\end{proof}


\section{Elliptic curves}
\label{S:elliptic-curves}

We investigate elliptic curves as an application of degeneration arguments to interpolation. Instead of stating the most general result first, our philosophy is to start with a simpler one, identify the points which can be improved, and build towards stronger statements gradually.

\begin{corollary}
  \label{T:elliptic-normal-curves-interpolation}
  A general elliptic normal curve of degree $d \geq 7$ satisfies weak interpolation.
\end{corollary}

\begin{proof}
  The idea behind this proof is to study a degeneration consisting of a rational normal curve $C \subset \Pbb^{d-1}$ and a secant line $L$. The normal bundle of $C$ is perfectly balanced $N_C \cong \Oc_{\Pbb^1}(d+1)^{\oplus (d-2)}$ (see \cref{T:rnc-normal-bundle}), so $\H^1(N_C) = 0$ and it satisfies $((d-2)^{d+2})$-interpolation. \cref{T:curve-secant-interpolation} then implies that $X = C \cup L$ satisfies weak interpolation. In fact, $X$ satisfies the stronger $((d-2)^{d+2}, 2, 2)$-interpolation.

  The rank of the normal bundle $N_X$ is $d-2$ and $\h^0(N_X) = d^2$. Weak interpolation is equivalent to $((d-2)^{d+2})$-interpolation as long as
  \[
  \h^0(N_X) - (d-2)(d+2) < \rank N_X.
  \]
  After simplification the inequality becomes $d > 6$, hence the hypothesis on $d$.

  By \cref{T:curve-secant-smoothable} the curve $X$ is smoothable and $\H^1(N_X) = 0$ (hence $[X]$ is a smooth point of the corresponding Hilbert scheme). We covered the hypotheses of \cref{T:interpolation-open-Hilbert}, so the interpolation statement holds for a general elliptic normal curve.
\end{proof}

\begin{remark}
  \label{T:rnc-normal-bundle}
  \cref{T:skew-tangents} provides a strengthening of \cref{T:curve-secant-interpolation} which implies that a general elliptic normal curve satisfies weak interpolation with tail $(2)$. In fact, the high symmetry of rational normal curves implies that all pairs of points are alike, so we do not need to choose a general secant line.

  By studying the normal bundle of rational normal curves, we can extend the tail in question to $(2,2)$. Let $C \subset \Pbb^{d-1}$ be a rational normal curve. The splitting of its normal bundle $N_C = N_{C/\Pbb^{d-1}} \cong \Oc_{\Pbb^1}(d+1)^{\oplus(d-2)}$ can be realized as follows. Pick distinct points $q_1, \dots, q_{d-2}$ on $C$. For each $q_j$ we consider the cone $S_j$ of $C$ over $q_j$. It is clear that $S_j$ is smooth away from $q_j$, so the normal bundle $N_j = N_{C/S_j}$ is a line bundle over $C \setminus q_j$ sitting in $N_C$. The curve-to-projective extension theorem \cite[I.6.8]{Hartshorne} allows us to extend $N_j$ and its inclusion over $q_j$. At a point $q \neq q_j$, the fiber of $N_j$ corresponds to the 2-plane through $q_j$ which contains the tangent line to $q$. When $q = q_j$ the fiber of $N_j$ comes from the osculating plane to $C$ at $q_j$. By studying the bundles $N_j$ in more detail one can show that each has degree $d+1$, and they furnish a splitting $N_C \cong \bigoplus_{j=1}^{d-2} N_j \cong \Oc_{\Pbb^1}(d+1)^{\oplus(d-2)}$.

  Let $L$ be a secant line $L$ to $C$ through two of the points $q_j$, say $q_1$ and $q_2$. We consider a modification $N = N_{C \cup L}|_C(-\sum_i p_i)$ given by $d+2$ general points $p_i$ on $C$. To show that $C \cup L$ satisfies weak interpolation with tail $(2,2)$ it suffices to show that $N \cong \Oc_{\Pbb^1}(-1)^{\oplus d} \oplus \Oc_{\Pbb_1}^{\oplus 2}$ and not $N \cong \Oc_{\Pbb^1}(-1)^{\oplus(d+1)} \oplus \Oc_{\Pbb^1}(1)$. This follows from the fact that $N_1 \oplus N_2$ is a subbundle of $N_C$.
\end{remark}

It is possible to generalize the argument from \cref{T:elliptic-normal-curves-interpolation} to a much broader class of elliptic curves.

\begin{theorem}
  \label{T:elliptic-curves-interpolation}
  A general elliptic curve of degree $d$ in $\Pbb^n$ satisfies weak-interpolation if
  \begin{enumeratea}
  \item
    $d \geq 7$,
  \item
    $d \geq n+1$, and
  \item
    the remainder of $2d$ modulo $n-1$ lies between $3$ and $n-2$ inclusive.
  \end{enumeratea}
\end{theorem}

\begin{figure}[h!]
  \centering
  \includegraphics[width=0.75\textwidth]{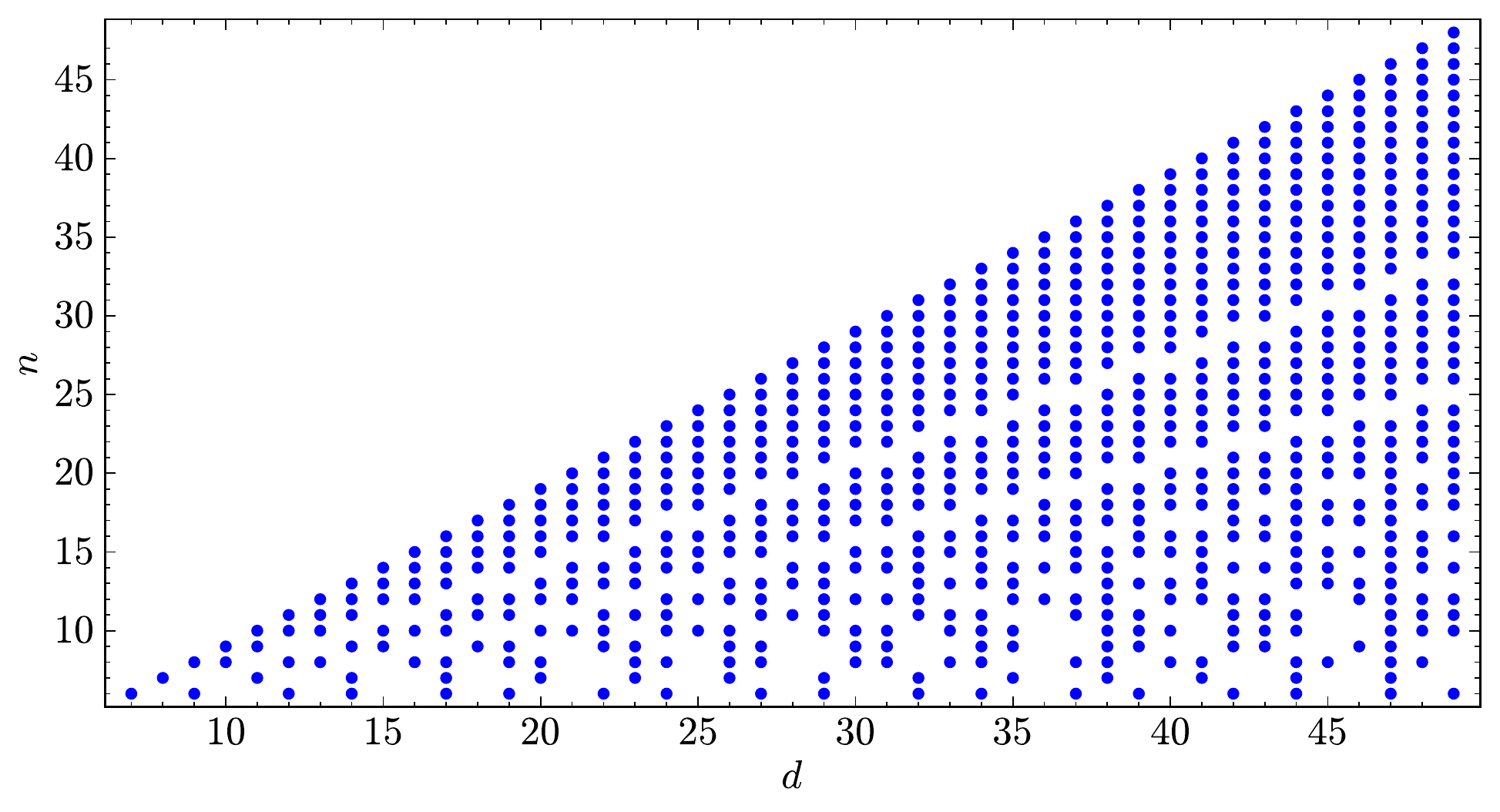}
  \caption{Values of $d$ and $n$ which satisfy the hypotheses of \cref{T:elliptic-curves-interpolation}}
  \label{F:elliptic-interpolation}
\end{figure}

The proof of \cref{T:elliptic-curves-interpolation} rests on the following result.

\begin{theorem}[Ran \cite{Ran}]
  \label{T:rational-curves-interpolation}
  A general spanning rational curve satisfies interpolation.
\end{theorem}

\begin{remark}
  \label{T:rational-curves-interpolation-spanning}
  A general rational $C \subset \Pbb^n$ of degree $d$ is spanning as long as $d \geq n$.
\end{remark}

\begin{remark}
  \label{T:rational-curve-splitting}
  Let us derive the precise splitting type of a general degree $d$ rational curve $C \subset \Pbb^n$. We assume $d \geq n$. The Euler sequence restricted to $C$ reads
  \[\xymatrix{
    0 \ar[r] &
    \Oc_C \ar[r] &
    \Oc_C(1)^{\oplus(n+1)} \ar[r] &
    T_{\Pbb^n}|_C \ar[r] &
    0,
  }\]
  so we have equality of degrees
  \[
  \deg T_{\Pbb^n}|_C =
  \deg \Oc_C^{\oplus(n+1)} =
  (n+1) d.
  \]
  The normal bundle of $C$ sits in the sequence
  \[\xymatrix{
    0 \ar[r] &
    T_C \ar[r] &
    T_{\Pbb^n}|_C \ar[r] &
    N_C \ar[r] &
    0,
  }\]
  and we deduce
  \[
  \deg N_C =
  \deg T_{\Pbb^n}|_C - \deg T_C =
  (n+1)d - 2.
  \]
  If we write $2d-2 = q \cdot (n-1) + r$ and $0 \leq r < n-1$, then \cref{T:rational-curves-interpolation} implies
  \[
  N_C \cong
  \Oc_{\Pbb^1}(d + q + 1)^{\oplus r} \oplus \Oc_{\Pbb^1}(d + q)^{\oplus(n-1-r)}.
  \]
\end{remark}

\begin{proof}[Proof of \cref{T:elliptic-curves-interpolation}]
  We consider the union of a general degree $d-1$ rational curve $C$ and a secant line $L$. Following \cref{T:rational-curve-splitting}, we deduce
  \[
  N_C = \Oc_{\Pbb^1}(d + q)^{\oplus r} \oplus \Oc_{\Pbb^1}(d + q - 1)^{\oplus(n-1-r)}
  \]
  where $2d-4 = q \cdot (n-1) + r$ for $0 \leq r < n-1$. \cref{T:curve-secant-interpolation} implies $C \cup L$ satisfies $((n-1)^{d+q})$-interpolation. This is the definition of weak interpolation for $C \cup L$ as long as
  \[
  \left\lfloor \frac{\h^0(N_C)}{\rank N_C} \right\rfloor =
  \left\lfloor \frac{\h^0(N_{C \cup L})}{\rank N_{C \cup L}} \right\rfloor.
  \]
  We know that $\h^0(N_{C \cup L}) = h^0(N_C) + 4$, so we can rewrite the equality as
  \[
  d + q = d + q + \left\lfloor \frac{r + 4}{n-1} \right\rfloor.
  \]
  Since $2d-4 = q \cdot (n-1) + r$, the condition $r + 4 < n-1$ is equivalent to the cases
  \[
  2d \equiv 4, \dots, n-2 \pmod{n-1}.
  \]

  It remains to show our claim holds when $2d \equiv 3 \pmod{n-1}$. Consider the subbundle
  \[
  E = \Oc_{\Pbb^1}(d+q)^{\oplus r} \subset N_C
  \]
  consisting of all pieces of the larger degree $d+q$. Alternatively, we can construct $E$ as the largest proper subbundle in the Harder-Narasimhan filtration of $N_C$. Our choice of $r$ ensures $E$ is proper of corank $1$. \cref{T:general-secant-away} shows that the general secant is away from $E$. The decomposition of $N_{C \cup L}|_C$ is obtained from the decomposition of $N_C$ by adding $1$ to the degrees of two of the components (possibly the same one). By choosing $L$ to be away from $E$, we are ensuring that the single component of lower degree $d-q-1$ is bumped up, so $C \cup L$ satisfies interpolation.

  The statement about weak interpolation of a general elliptic curve follows by applying \cref{T:interpolation-open-Hilbert}.
\end{proof}

\begin{remark}
  The congruence condition on $2d$ mod $n-1$ excludes all cases when $n = 3$ and $4$. We will later show that the general elliptic curve in $\Pbb^3$ satisfies interpolation as long as $d \geq 4$ (see \cref{T:curves-interpolation-P3}).
\end{remark}


\section{Curves in $\Pbb^3$}
\label{S:curves-in-P3}

The study of space curves in algebraic geometry has a long history and there are many strong results we can rest on. This makes for an exciting application of interpolation ideas. We aim to show the following result.

\begin{theorem}
  \label{T:curves-interpolation-P3}
  A general connected genus $g$ and degree $d$ curve in $\Pbb^3$ satisfies interpolation if $d \geq g + 3$ unless $d = 5$ and $g = 2$.
\end{theorem}

Our proof of \cref{T:curves-interpolation-P3} depends on the following analysis of secant lines to rational curves in $\Pbb^3$.

\begin{proposition}
  \label{T:rational-curve-P3-secants}
  Let $C \subset \Pbb^3$ be a general rational curve of degree $d \geq 3$. Consider $k$ general secant lines $L_1, \dots, L_k$ and the union curve $X = C \cup \bigcup_i L_i$. Unless $d = 3$ and $k = 2$, the vector bundle $N_X|_C \cong \Oc_{\Pbb^1}(2d - 1 + k)^{\oplus 2}$ is balanced. In the exceptional case $d = 3$, $k = 2$, the bundle is $N_X|_C \cong \Oc_{\Pbb^1}(6) \oplus \Oc_{\Pbb^1}(8)$.
\end{proposition}

Let us assume this result and deduce \cref{T:curves-interpolation-P3} from it.

\begin{proof}[Proof of \cref{T:curves-interpolation-P3}]
  The normal bundle of a space curve has rank $2$, so we can focus on demonstrating weak interpolation by \cref{T:strong-interpolation-2}. Ran's \cref{T:rational-curves-interpolation} guarantees that a general rational curve $C$ of degree $d-g \geq 3$ has balanced normal bundle, so it satisfies weak interpolation. Let $L_1, \dots, L_g$ be $g$ secant lines to $C$ and set $X = C \cup L_1 \cup \cdots \cup L_g$. By choosing $L_i$ general, we can assume
  \begin{enumeratea}
  \item
    all $L_i$ are disjoint (by \cref{T:skew-secants}),
  \item
    the tangent lines to $C$ at $C \cap (L_1 \cup \cdots \cup L_g)$ are all skew (by \cref{T:skew-tangents}), and
  \item
    the vector bundle $N_X|_C$ is balanced (by \cref{T:rational-curve-P3-secants}).
  \end{enumeratea}
  We apply \cref{T:curve-secant-lines-smoothable} to conclude that $X$ is smoothable and $\H^1(N_X) = 0$. In other words, $[X]$ is a smooth point of the Hilbert scheme $\Hc_{d,g,3}$. Since interpolation is an open condition (\cref{T:interpolation-open-Hilbert}), it suffices to show $X$ satisfies weak interpolation. We do this by twisting back $N_X$ at $2d - g$ points $q_j$ on $C$ and one point $p_i$ on each $L_i$. We will choose all $p_i$ and $q_j$ to be smooth points of $X$. We have a short exact sequence
  \[\xymatrix{
    0 \ar[r] &
    N_X(-\sum_i p_i - \sum_j q_j) \ar[r] &
    N_X|_C(-\sum_j q_j) \oplus \bigoplus_i N_X|_{L_i}(-p_i) \ar[r] &
    N_X|_{\Sing X} \ar[r] &
    0.
  }\]
  Since the tangent lines of $C$ to all nodal points $\Sing X$ are skew, it follows that $N_X|_{L_i} \cong \Oc_{L_i}(2)^{\oplus 2}$ and $N_X|_{L_i}(-p_i) \cong \Oc_{L_i}(1)^{\oplus 2}$. This ensures that the morphism
  \[\xymatrix{
    \bigoplus_i N_X|_{L_i}(-p_i) \ar[r] &
    N_X|_{\Sing X} = \bigoplus_i N_X|_{L_i \cap C}
  }\]
  is surjective on $\H^0$. But $N_X|_C$ is balanced, so $\H^1(N_X|_C(-\sum_j q_j)) = 0$ and $\H^1(N_X(-\sum_i p_i - \sum_j q_j)) = 0$. This completes the proof of the desired statement.
\end{proof}

The proof of \cref{T:rational-curve-P3-secants} rests on understanding elementary modifications of rank $2$ bundles. This topic is very classical, so our discussion will be succinct. We refer to \cite{Beauville} and \cite{Reid-surfaces} for a detailed discussion of ruled and rational surfaces.

\begin{example}[Elementary modifications in rank $2$]
  \label{T:rank-2-modifications}
  Let $C$ be a smooth curve and $E$ a rank $2$ bundle over it. We consider a modification at a point $p \in C$ given by a line $L \subset E_p$:
  \[\xymatrix{
    0 \ar[r] &
    E' = M(E, L) \ar[r] &
    E \ar[r] &
    E_p/L \ar[r] &
    0.
  }\]
  The line $L$ corresponds to a point $\ell \in \Pbb E$ lying over $p$. If $\pi \cn \Pbb E \rarr C$ is the natural projection morphism, let $F = \pi^{-1}(p)$ denote the fiber containing $\ell$. The blow-up $S = \Bl_\ell \Pbb E$ of $\Pbb E$ at $\ell$ is a smooth surface with exceptional locus $E$. Let $\wtilde{F}$ denote the proper transform of $F$ in $S$. Since $E^2 = -1$ and $(\wtilde{F} + E)^2 = F^2 = 0$, it follows that $\wtilde{F}^2 = -1$. We can contract $\wtilde{F}$, and the resulting surface is a $\Pbb^1$-bundle over $C$, naturally isomorphic to $\Pbb E'$. In general, every ruled surface can be presented as $\Pbb E$ for some $C$ and $E$.

  Note that the image of the contracted curve $\wtilde{F}$ is a distinguished point on $\Pbb E'$. The process is symmetric, that is, if we blow up this point in $\Pbb E'$ and blow down the proper transform of the fiber we arrive back at $\Pbb E$.

  It should not come as a surprise that the operation we described above is called an \emph{(elementary) modification} of a ruled surface. If $S$ is a ruled surface and $p$ a point on it, we will use $M(S, p)$ to denote the resulting modification and $\wtilde{p} \in M(S,p)$ the distinguished point corresponding to $p$. More generally, given a list of points $p_1, \dots, p_k$ lying in distinct fibers, we can consider the simultaneous modification $M(S, p_1, \dots, p_k)$. The construction furnishes a birational equivalence $m_{S,p_1, \dots, p_k} \cn S \drarr M(S, p_1, \dots, p_k)$. When there is no ambiguity, we may shorten our notation $m_{S, p_i}$ to $m_{p_i}$.

  Let us specialize to $C = \Pbb^1$. Since $\Pbb(\Oc(a) \oplus \Oc(b)) \cong \Pbb(\Oc \oplus \Oc(b-a))$, it suffices to study the \emph{Hirzebruch surfaces} $\Fbb_n = \Pbb(\Oc \oplus \Oc(n))$ for $n \geq 0$. In $\Fbb_0 \cong \Pbb^1 \x \Pbb^1$ there is a unique $0$-section through each point on the surface; there are no negative sections. When $n > 0$, the surface $\Fbb_n$ has a distinguished section $B_n$ with self-intersection $B_n^2 = -n$. Our goal is to describe the modification at any point $p \in \Fbb_n$.
  \begin{description}
  \item[Case 1, $n = 0$.]
    All points on $p \in \Fbb_0$ are alike, and any modification leads to an $\Fbb_1$. The $(-1)$-section is the proper transform of the $0$-section through $p \in \Fbb_0$.

  \item[Case 2, $n > 0$.]
    If $p \in B_n$, then the resulting surface is $\Fbb_{n+1}$. Its distinguished section is the proper transform of $B_n$. If $p \notin B_n$, then we arrive at $\Fbb_{n-1}$. The section $B_n$ maps isomorphically to the distinguished section of the new surface.
  \end{description}

  The following diagram summarizes the behavior of Hirzebruch surfaces under elementary modifications. Solid arrows represent general modifications while dotted ones stand for modifications based at a point on the distinguished section. 

  \hspace{\baselineskip}
  \[\xymatrix@C=0.5in{
    \Fbb_0 \ar@{=>}@/_1pc/[r]&
    \Fbb_1 \ar@{.>}@/_1pc/[r] \ar@{=>}@/_1pc/[l] &
    \Fbb_2 \ar@{.>}@/_1pc/[r] \ar@{=>}@/_1pc/[l] &
    \Fbb_3 \ar@{.>}@/_1pc/[r] \ar@{=>}@/_1pc/[l] &
    \cdots \ar@{=>}@/_1pc/[l] 
  }\]
  \hspace{\baselineskip}

  Let us investigate the result of repeated modifications applied to a Hirzebruch surface $\Fbb_n$ where $n > 0$. If the first point $p_1$ is general (away from the $(-n)$-section), then $M(\Fbb_n, p_1) \cong \Fbb_{n-1}$. In choosing the second point $p_2 \in \Fbb_n$, we will require that it is not in the fiber of $p_1$ and away from $B_n$ (the proper transform of $B_{n-1} \subset M(\Fbb_n, p_1)$ in $\Fbb_n$). Under these conditions $M(\Fbb_n, p_1, p_2) \cong \Fbb_{n-2}$. Continuing this argument we can show that if the first $n$ points are taken to be general, then $M(\Fbb_n, p_1, \dots, p_n) \cong \Fbb_0$. Instead of adding more points, we will reset the process and present the case $n = 0$ separately.

  For any $p_1 \in \Fbb_0$, the resulting modification is $M(\Fbb_0, p_1) \cong \Fbb_1$. The $(-1)$-section in $M(\Fbb_0, p_1)$ is the transform of the horizontal section through $p_1$ in $\Fbb_0$. As long as $p_2$ avoids it as well as the fiber of $p_1$, we will have $M(\Fbb_0, p_1, p_2) \cong \Fbb_0$. In choosing the third point we avoid the fibers of the first two to obtain $M(\Fbb_0, p_1, p_2, p_3) \cong \Fbb_1$. The $(-1)$-section in this surface is the transform of the unique $2$-section through the first three points, so avoiding it we are back to $\Fbb_0$. It is clear how the process continues and the final result depends on the parity of the number of points we modify at. We demonstrated the following result.

  \begin{proposition}
    \label{T:modifying-Fn}
    If $p_1, \dots, p_k \in \Fbb_n$ are general points, then
    \[
    M(\Fbb_n, p_1, \dots, p_k) \cong
    \begin{cases}
      \Fbb_{n-k} & \textrm{if $k \leq n$}, \\
      \Fbb_0 & \textrm{if $k > n$ and $k - n \equiv 0 \pmod{2}$}, \\
      \Fbb_1 & \textrm{if $k > n$ and $k - n \equiv 1 \pmod{2}$}.
    \end{cases}
    \]
  \end{proposition}
\end{example}

Our discussion so far explored the effect of modifying at a single general point. What if we are interested in modifying at two points which are not general, i.e., there is a relation between them? We already saw some of these ideas, but we will take the opportunity to develop and state them more clearly. Our goal is to arrive at a more refined version of \cref{T:modifying-Fn}.

\begin{example}[Modifying at two points at a time]
  To review notation, $B_n$ will denote the distinguished $(-n)$-section in $\Fbb_n$ when $n > 0$. We will use $\pi_n \cn \Fbb_n = \Pbb(\Oc \oplus \Oc(n)) \rarr \Pbb^1$ to denote the natural projection. When $n = 0$ and $\Fbb_0 \cong \Pbb^1 \x \Pbb^1$, we will treat $\pi_0$ as the projection to the first factor.

  \begin{description}
  \item[Case 1, $n \geq 2$.]
    If we choose $p_1 \notin B_n$, then $M(\Fbb_n, p_1) \cong \Fbb_{n-1}$. In order to arrive at $\Fbb_{n-2}$, the second point should avoid the fiber of the first and the transform of $B_{n-1}$, that is $B_n$.
    
  \item[Case 2, $n = 1$.]
    The first point should avoid $B_1$, but the second only need to avoid the fiber of the first.
    
  \item[Case 3, $n = 0$.]
    There is no restriction on the first point; all modifications lead to $M(\Fbb_0, p_1) \cong \Fbb_1$. In choosing the second point, we would like to avoid the fiber of the first and the transform of the $(-1)$-curve which is the horizontal section through $p_1$.

    The horizontal sections on $M(\Fbb_0, p_1, p_2) \cong \Fbb_0$ are even more interesting. They form a pencil of $0$-sections. Tracing them back through $m_{p_1, p_2}$, their transforms in the original surface remain sections, pass through $p_1$ and $p_2$, and have self-intersection $2$ (see \cref{F:modification-self-intersection}). The only line bundle which fits the bill is $\Oc_{\Fbb_0}(1,1)$; its global sections form a $4$-dimensional vector space. Passing through each of the points $p_1$ and $p_2$ imposes a single linear condition, so indeed we are left with a pencil of $(1,1)$-curves passing through $p_1$ and $p_2$.
  \end{description}

  \newcommand{\arr}[2]{\ensuremath{\xymatrix@C=#1{\ar[r]^{#2}&}}}
  \newcommand{\darr}[2]{\ensuremath{\xymatrix@C=#1{\ar@{-->}[r]^{#2}&}}}

  \begin{figure}[h!]
    \centering
    \def\svgwidth{0.7\textwidth}
    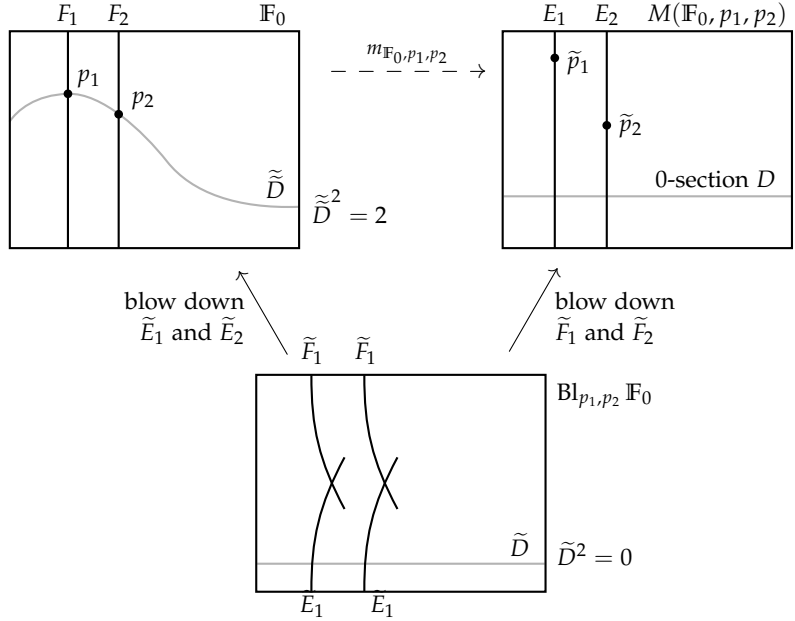
    \caption{A horizontal section $D$ in $M(\Fbb_0, p_1, p_2) \cong \Fbb_0$ transforms to a $(1,1)$-curve $\wtilde{\wtilde{D}}$ in $\Fbb_0$ with self-intersection $2$.}
    \label{F:modification-self-intersection}
  \end{figure}

  The pattern we described generalizes as follows.

  \begin{proposition}
    \label{T:modifying-F0}
    Let $p_1, \dots, p_{2k} \in \Fbb_0$ be points lying in distinct fibers of $\pi_0 \cn \Fbb_0 \rarr \Pbb^1$. If the modified surface $S = M(\Fbb_0, p_1, \dots, p_{2k})$ is isomorphic to $\Fbb_0$, then under $\Fbb_0 \drarr S$ the pencil of horizontal sections in $S$ transforms to the pencil of $(k,1)$-curves passing through all $p_i \in \Fbb_0$.
  \end{proposition}

  \begin{remark}
    \label{T:modifying-F0-final-modification}
    Note that we only impose a condition on the final modification $M(\Fbb_0, p_1, \dots, p_{2k})$. The result applies even if one of the intermediate modifications $M(\Fbb_0, p_1, \dots, p_i)$ is not isomorphic to $\Fbb_0$ or $\Fbb_1$.
  \end{remark}

  The result above will be useful to us in the following form.

  \begin{corollary}
    \label{T:modifying-F0-section}
    Let $p_1, \dots, p_{2k+2} \in \Fbb_0$ be points lying on distinct fibers of $\pi_0 \cn \Fbb_0 \rarr \Pbb^1$. Suppose the surface $M(\Fbb_0, p_1, \dots, p_{2k})$ is isomorphic to $\Fbb_0$. Then the surface $M(\Fbb_0, p_1, \dots, p_{2k+2})$ is isomorphic to $\Fbb_0$ if and only if $p_1, \dots, p_{2k+2}$ do not lie on a $(k,1)$-curve.
  \end{corollary}

  \begin{proof}[Proof of \cref{T:modifying-F0}]
    Horizontal sections on $S = M(\Fbb_0, p_1, \dots, p_{2k}) \cong \Fbb_0$ form a pencil and have self-intersection $0$. Reversing a modification involves blowing down $2k$ fibers, so the self-intersection of a transformed section increases by $1$ for each point of modification. It follows that the curves which correspond to horizontal sections in $S$ have self-intersection $2k$ on $\Fbb_0$. Since sections remain sections, these are curves on $\Fbb_0$ of the form $(k',1)$ for some $k' \geq 0$. Their self-intersection is $k' + k' = 2k$, so $k' = k$.

    Starting with a $(k,1)$-curve in $\Fbb_0$ passing through all $p_i$, we can reverse the argument and show that its transform in $S$ is a $0$-section. This completes the argument.
  \end{proof}
\end{example}

There is another useful result worth isolating from the proof of \cref{T:rational-curve-P3-secants}.

\begin{proposition}
  \label{T:different-leaves}
  Let $X, Y$ be normal varieties, $X$ is irreducible, and $f, g \cn X \rarr Y$ are two regular morphisms. Consider a pencil $\{ H_t \}_{t \in \Pbb^t}$ on $Y$ with base locus $B \subset Y$ and associated line bundle $L$. If
  \begin{enumeratea}
  \item
    $f^\ast L$ is not isomorphic to $g^\ast L$, and
  \item
    the codimension of $f^{-1}(B) \cup g^{-1}(B) \subset X$ is at least two,
  \end{enumeratea}
  then for a general point $x \in X$ the images $f(x), g(x) \in Y$ lie in distinct members of the pencil $\{ H_t \}$.
\end{proposition}

\begin{proof}
  Since $Y$ is normal, the base locus $B$ of $\{ H_t \}$ has codimension at least two. The pencil has an associated map $\phi \cn V \rarr \Pbb^1$ where $V = Y \setminus B$. Let $U = X \setminus (f^{-1}(B) \cup g^{-1}(B))$. We will use $i_U \cn U \rarr X$ and $i_V \cn V \rarr Y$ to denote the natural inclusions.
  \[\xymatrix{
    U \ar@<0.5ex>[r]^-{f|_U} \ar@<-0.5ex>[r]_-{g|_U} \ar[d]_-{i_U} & V \ar[r]^-{\phi} \ar[d]^-{i_V} & \Pbb^1 \\
    X \ar@<0.5ex>[r]^-{f} \ar@<-0.5ex>[r]_-{g} & Y
  }\]
  Since the complements of both $U$ and $V$ are of codimension at least two, it follows that $i_U^\ast \cn \Pic(X) \rarr \Pic(U)$ and $i_V^\ast \cn \Pic(Y) \rarr \Pic(Y)$ are both isomorphisms. By construction $\phi^\ast \Oc_{\Pbb^1}(1) \cong i_V^\ast L$. Combining this with the hypothesis $f^\ast L \not\cong g^\ast L$, it follows that
  \[
  (\phi \circ f|_U)^\ast \Oc_{\Pbb^1}(1) \not\cong (\phi \circ g|_U)^\ast \Oc_{\Pbb^1}(1).
  \]
  This demonstrates that $\phi \circ f|_U$ and $\phi \circ g|_U$ are not identical, hence proving our claim.
\end{proof}

We are now ready for the central argument of this section.

\begin{proof}[Proof of \cref{T:rational-curve-P3-secants}]
  Our first goal is to translate the question into the language of rational surfaces. Consider the rational morphism $f \cn C \x C \drarr \Pbb N_C$ which sends the point $(p,q) \in C \x C$ to $T_p [p,q]$, the tangent space to the secant line joining $p$ and $q$, viewed as a point of $\Pbb N_{C,p}$. This definition makes sense when $p \neq q$ and the secant $[p,q]$ is not tangent at $p$. Trisecant lines which are also tangent, called \emph{tangential trisecants}, have been a classical topic of interest. If $C$ is a general rational curve, then using osculating planes $f$ extends to the entire diagonal and the number of problematic points, corresponding to tangential trisecants, is $2(d-2)(d-3)$ (see \cref{T:tangential-trisecants}).

  We would like to resolve $f$ and $f \circ \sigma$ simultaneously, where $\sigma \cn C \x C \rarr C \x C$ is the \emph{swap} morphism given by $\sigma(p,q) = (q,p)$. Let $s_1, \dots, s_{2(d-2)(d-3)} \in C \x C$ comprise the complement of the locus where $f$ is well-defined, and let $s_i' = \sigma(s_i)$ be their swaps. The blow-up $\epsilon \cn S = \Bl_{\{s_i, s_i'\}} C \x C \rarr C \x C$ is the minimal simultaneous resolution of $f$ and $f \circ \sigma$. If $\wtilde{\sigma} \cn S \rarr S$ is the lift of $\sigma$ and $\wtilde{f} \cn S \rarr \Pbb N_C$ is the resolution of $f$, then the resolution of $f \circ \sigma$ is $\wtilde{f} \circ \wtilde{\sigma}$.

  \[\xymatrix{
    & S \ar[d]^-{\epsilon} \ar[dl]_-{\wtilde{f}} \ar[dr]^-{\wtilde{f} \circ \wtilde{\sigma}} \\
    \Pbb N_C & C \x C \ar@{-->}[l]^-{f} \ar@{-->}[r]_-{f \circ \sigma} & \Pbb N_C
  }\]

  Note that to resolve $f$ we only need to blow up $C \x C$ at isolated points. This follows since the morphism $f \cn C \x C \drarr \Pbb N_C$ is given by two pencils, one of them (the projection on the first factor) has no base locus, and the second one has reduced base locus for general $C$ (see \cite[p.~291]{Griffiths-Harris} and \cite{Johnsen-plane-projections}).

  Let us compute the degree of $f \cn C \x C \drarr \Pbb N_C$ (and hence of $\wtilde{f}$). A point on $\Pbb N_C$ corresponds to a point $p \in C$ and a line in the normal space $N_{C,p}$. The preimage of this line under the quotient morphism $T_{\Pbb^3,p} \rarr N_{C,p}$ yields a plane in the tangent space at $p$ containing the line $T_p C$. We can treat the 2-plane in the tangent space as an actual plane $H \subset \Pbb^3$. The degree problem amounts to counting the number of secant lines to $C$ contained in $H$. We already know that $H$ contains $p \in C$. We have the freedom to pick a general $H$ as long as $T_p C \subset T_p H$. Since $H$ is tangent to $C$ and the point $p$ is not a flex of $C$, the intersection $H \cap C$ has multiplicity $2$ at $p$. The curve $C$ has degree $d$, hence $H$ intersects $C$ transversely in $d-2$ points away from $p \in C$ (see \cref{F:P3-rational-degree}). This demonstrates that $\deg f = \deg \wtilde{f} = d-2$. When $d = 3$ the map $f \cn C \x C \rarr \Pbb N_C$ is an isomorphism.

  \begin{figure}[h!]
    \centering
    \def\svgwidth{0.5\textwidth}
    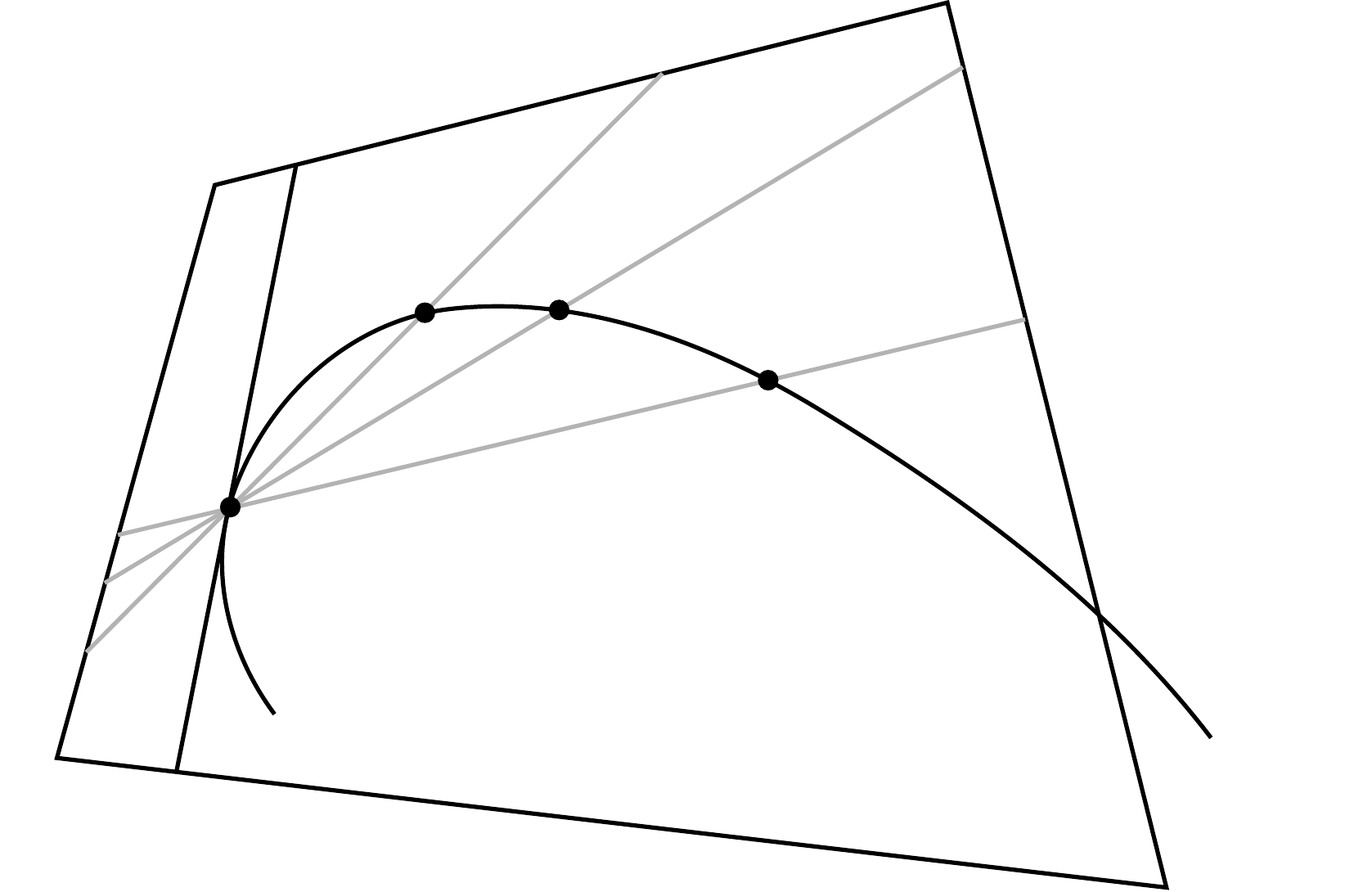
    \caption{Computing the degree of $f \cn C \x C \drarr \Pbb N_C$}
    \label{F:P3-rational-degree}
  \end{figure}
  
  Our next goal is to determine how line bundles behave under pullbacks via $\wtilde{f}$ and $\wtilde{f} \circ \wtilde{\sigma}$. Set $E = \epsilon^{-1}(\{ s_i \})$ and $E' = \epsilon^{-1}(\{ s_i' \}) = \wtilde{\sigma}(E)$. Since $C$ is rational, all line bundles on $C \x C$ are of the form $\Oc_{C \x C}(a,b)$. For convenience, we will denote $\Oc_S(a,b) = \epsilon^\ast \Oc_{C \x C}(a,b)$. It is easy to characterize $\wtilde{\sigma}$ using this notation:
  \[
  \wtilde{\sigma}^\ast \Oc_S(a,b) = \Oc_S(b,a), \qquad
  \wtilde{\sigma}^\ast \Oc_S(E) = \Oc_S(E'), \qquad
  \wtilde{\sigma}^\ast \Oc_S(E') = \Oc_S(E).
  \]
  Since we can choose $C$ to be general, its normal bundle $N_C$ is balanced (\cref{T:rational-curves-interpolation}) and $\Pbb N_C \cong C \x \Pbb^1 \cong \Fbb_0$. Since $f$, and hence $\wtilde{f}$, respect the projection to the first factor, the pullback of a fiber is a fiber and $\wtilde{f}^\ast \Oc_{\Pbb N_C}(1,0) = \Oc_S(1,0)$. By the construction of $\wtilde{f}$, it follows that $\wtilde{f}^\ast \Oc_{\Pbb N_C}(0,1) = \Oc_S(a,b)(-E)$ for some $a,b \in \Zbb$. To compute $a$ and $b$, we use the intersection forms on $\Pic S$ and $\Pic \Pbb N_C$. Consider a point in $\Pbb N_C$ as the intersection of a $(0,1)$ and a $(1,0)$ fiber.
  \begin{align*}
    b
    &= \Oc_S(1,0) \cdot \Oc_S(a,b)(-E) \\
    &= \wtilde{f}^\ast \Oc_{\Pbb N_C}(1,0) \cdot \wtilde{f}^\ast \Oc_{\Pbb N_C}(0,1) \\
    &= (\deg\wtilde{f}) \left( \Oc_{\Pbb N_C}(1,0) \cdot \Oc_{\Pbb N_C}(0,1) \right) \\
    &= d-2
  \end{align*}
  The self-intersection of a horizontal section in $\Pbb N_C$ allows us to compute $a$.
  \begin{align*}
    0
    &= (\deg\wtilde{f}) (\Oc_{\Pbb N_C}(0,1)^2) \\
    &= (\wtilde{f}^\ast \Oc_{\Pbb N_C}(0,1))^2 \\
    &= \Oc_S(a,d-2)(-E)^2 \\
    &= 2 a (d-2) + E^2
  \end{align*}
  Since $E$ consists of $2(d-2)(d-3)$ $(-1)$-curves, it follows that $a = d-3$. In summary of our results
  \begin{align*}
    \wtilde{f}^\ast \Oc_{\Pbb N_C}(i, j)
    &= \Oc_S(i + (d-3)j, (d-2)j)(-j E), \\
    (\wtilde{f} \circ \wtilde{\sigma})^\ast \Oc_{\Pbb N_C}(i, j)
    &= \Oc_S((d-2)j, i + (d-3)j)(-j E')
  \end{align*}
  for all $i, j \in \Zbb$.
  
  The reason we constructed $f$ and $\sigma$ is the following. Let $L_1, \dots, L_k$ be secants to $C$ corresponding to the points $\ell_i = (a_i, b_i) \in C \x C \setminus (\Delta \cup \{ s_i, s_i' \})$. Then the projectivization of $N_{C \cup L_1 \cup \cdots \cup L_k}|_C$ is isomorphic to the modification of $\Pbb N_C$ at the images of $s_i$ under $f$ and $f \circ \sigma$. Our goal is to show that for general points $s_i$ we have an isomorphism
  \[
  M(\Pbb N_C, f(\ell_1), f \circ \sigma(\ell_1), \dots, f(\ell_k), f \circ \sigma(\ell_k)) \cong
  \Fbb_0.
  \]

  We will use induction on $k$. At each step, given that $\wtilde{f}^\ast \Oc_{\Pbb N_C}(k,1) \not\cong (\wtilde{f} \circ \wtilde{\sigma})^\ast \Oc_{\Pbb N_C}(k,1)$ we will apply \cref{T:modifying-F0-section,T:different-leaves}. Our description of pullbacks above implies that
  \begin{align*}
    \wtilde{f}^\ast \Oc_{\Pbb N_C}(k,1)
    &= \Oc_S(k+d-3, d-2)(-E), \\
    (\wtilde{f} \circ \wtilde{\sigma})^\ast \Oc_{\Pbb N_C}(k,1)
    &= \Oc_S(d-2, k+d-3)(-E').
  \end{align*}

  \begin{description}
  \item[Case 1, $d \geq 4$.]
    Recall that $E$ consists of $2(d-2)(d-3) > 0$ curves. Since $E$, $E'$, and $\epsilon^\ast \Pic(C \x C)$ are all linearly independent in $\Pic(S)$, it follows that the line bundles above are not isomorphic. This holds for all $k \geq 0$, so we have covered both the base and inductive cases.
    
  \item[Case 2, $d = 3$.]
    When $C \subset \Pbb^3$ is a twisted cubic, $f$ is an isomorphism so both $E$ and $E'$ are empty. For convenience, we will identify $C \x C$ and $\Pbb N_C$.

    The two line bundles above are non-isomorphic except when $k = 1$. We can summarize the modifications of $\Pbb N_C$ as follows. Choosing a first point $\ell_1 \in C \x C$ away from the diagonal leads to a modification $M(\Pbb N_C, \ell_1, \sigma(\ell_1)) \cong \Fbb_0$. No matter where we choose the second point $\ell_2$, the four points $\ell_1, \sigma(\ell_1), \ell_2, \sigma(\ell_2)$ always lie on a $(1,1)$-curve $D$ and $\Pbb N_{C \cup L_1 \cup L_2}|_C \cong \Fbb_2$. Fortunately, this makes the choice of the third point quite easy. As long as $\ell_3$ avoid the fibers of the previous four points and the curve $D$ (the transform of the $(-2)$-curve in $\Pbb N_{C \cup L_1 \cup L_2}|_C$), the resulting modification will be balanced. At this point, the induction argument goes through analogously to the case $d \geq 4$. \qedhere
  \end{description}
\end{proof}
  
\begin{remark}
  \label{T:tangential-trisecants}
  We refer to \cite[Chapters 2.4 and 2.5]{Griffiths-Harris} for a more rigorous treatment of the following material. For a discussion of the general secant plane formula, see \cite[VIII.4]{ACGH}.
  
  Let $C \subset \Pbb^3$ be a smooth curve and $p \in C$ a point. We call $H \subset \Pbb^3$ an \emph{osculating plane} to $C$ at $p$ if $H$ and $C$ have order of contact $3$ at $p$. If the tangent line $T_p C$ at $p$ has order of contact $2$ with $C$ (which is the general case in characteristic $0$), then there is a unique osculating plane through $p$. On the other hand, if $C$ meets $T_p C$ with order $3$ or higher, then every plane containing $T_p C$ meets the criteria, so there is no unique osculating plane.

  There is an alternative way to construct the osculating plane at a point $p$. Consider two other points $q, r \in C$ and the plane $H_{q,r}$ spanned by $p$, $q$, and $r$. If a unique osculating plane exists, it is the limit of $H_{q,r}$ as both $q$ and $r$ approach $p$. Since an osculating plane contains the tangent line $T_p C$, it corresponds to a point in $\Pbb N_C$. This comment allows us to extend $f \cn C \x C \drarr \Pbb N_C$ to points $(a,a)$ on the diagonal as long as $a \in C$ is not a flex.

  We claim that a general degree $d \geq 3$ rational curve $C \subset \Pbb^3$ has no flex points. Since all points on the twisted cubic are alike, this is certainly the case for $d = 3$. For higher values of $d$, we note that every smooth rational curve is the image of a rational normal curve $C' \subset \Pbb^d$ away from a $(d-4)$-plane $\Lambda \subset \Pbb^d$. Let $\pi \cn \Bl_\Lambda \Pbb^d \rarr \Pbb^3$ denote the projection morphism. Every point of $p' \in C'$ has a well-defined osculating plane $H_{p'}$ and these sweep out a $3$-fold $X \subset \Pbb^d$. A point $p \in C$ is a flex if and only if the $\Lambda$ intersects $H_{p'}$ nontrivially where $p'$ is the unique point of $C'$ lying above $p$. In general the intersection $\Lambda \cap H_{p'}$ is a point, and $T_p C$ is the image of $H_{p'} \subset \Pbb^d$. Since $\Lambda$ and $X$ have complementary dimensions, we can choose $\Lambda$ away from $X$. It follows that the corresponding projection $C$ will have no flex points.
  
  Our analysis of flex points demonstrates that given a general rational curve $C$, we can define $f \cn C \x C \drarr \Pbb N_C$ along the diagonal. Tangential trisecants are the actual obstacle to extending $f$ to a regular morphism. Unfortunately, as soon as $d \geq 4$ they always exist. To illustrate our point, let us consider $d = 4$. A rational quartic curve $C$ lies on a smooth quadric surface $Q \subset \Pbb^3$ as a $(1,3)$-curve. Bezout's theorem implies that all trisecants to $C$ must be contained in $Q$. In fact, they are precisely the lines from one of the rulings on $Q$ (to be precise, the pencil of $(1,0)$-curves). Contracting the lines of the trisecant ruling produces a $3$-sheeted morphism $C \rarr \Pbb^1$. As long as $C$ has no flex points, this morphism is simply ramified. The Riemann-Hurwitz formula implies there are exactly $4$ simple ramification points, each corresponding to a tangential trisecant.

  While the method we presented for $d = 4$ does not generalize, there is another approach to counting tangential trisecants via what is called the trisecant correspondence (see \cite[p.~290-294]{Griffiths-Harris}). In modern language, this is a curve $D \subset C \x C$ defined as the closure of points $(p,q)$ such that $p \neq q$ and the line $[p,q]$ is trisecant to $C$. By projecting $C$ away from a general point on it and counting the number of nodes on the image curve in $\Pbb^2$ we can compute the projection degrees
  \[
  \deg(\pr_1 \cn D \rightarrow C) =
  \deg(\pr_2 \cn D \rightarrow C) =
  (d-2)(d-3).
  \]
  Tangential trisecants correspond to points in $D \cap \Delta$ where $\Delta \subset C \x C$ is the diagonal. In conclusion, the number of tangential trisecants is
  \[
  \#(D \cap \Delta) =
  \deg(\pr_1) + \deg(\pr_2) =
  2(d-2)(d-3).
  \]
\end{remark}

\begin{remark}
  \label{T:quintic-genus-2}
  Just because the union of a twisted cubic and two secant lines does not satisfy interpolation does not mean a priori that the general genus $2$ quintic curve indeed does not satisfy interpolation. We would like to fill this gap by demonstrating that a general such curve does not satisfy the full $(2^{10})$-interpolation, but it satisfies the slightly weaker $(2^9)$-interpolation.

  A general genus $2$ quintic curve $C \subset \Pbb^3$ lies on a smooth quadric surface $Q \subset \Pbb^3$ as a $(2,3)$-curve. The normal bundle sequence for the inclusion $C \subset Q \subset \Pbb^3$ reads
  \[\xymatrix{
    0 \ar[r] &
    N_{C/Q} \ar[r] &
    N_{C/\Pbb^3} \ar[r] &
    N_{Q/\Pbb^3}|_C \ar[r] &
    0,
  }\]
  which simplifies to
  \[\xymatrix{
    0 \ar[r] &
    \Oc_Q(2,3)|_C \ar[r] &
    N_{C/\Pbb^3} \ar[r] &
    \Oc_C(2) \ar[r] &
    0.
  }\]
  The first and third member of the sequence are line bundles of degrees $12$ an $10$ respectively. To see that $C$ does not satisfy $(2^{10})$-interpolation, we use the Jacobi inversion theorem (see \cite[I.3]{ACGH} and \cite[Chapter 2.2]{Griffiths-Harris}) to choose a degree $10$-divisor $D = \sum_{i=1}^{10} p_i$ such that $\Oc_Q(2,3)|_C(-D)$ is non-special (i.e., non-hyperelliptic) and $\Oc_C(2)(-D) \not\cong \Oc_C$. Then
  \[
  \h^0(\Oc_Q(2,3)|_C(-D)) = 1
  \quad\textrm{and}\quad
  \h^0(\Oc_C(2)(-D)) = 0,
  \]
  so
  \[
  \H^0(N_{C/\Pbb^3}(-D)) \cong
  \H^0(\Oc_Q(2,3)|_C(-D))
  \]
  is $1$-dimensional which is higher than the expected $0$.

  If $D$ were a general (effective) divisor of degree $9$, then the residual line bundle to $\Oc_Q(2,3)|_C(-D)$ has negative degree $-1$, so
  \[
  \h^0(N_{C/\Pbb^3}) =
  \h^0(\Oc_Q(2,3)|_C(-D)) + \h^0(\Oc_C(2)(-D)).
  \]
  As long as $D$ is such that $\Oc_C(2)(-D)$ is non-special, then $\h^0(N_{C/\Pbb^3}) = 2$ as expected.

  There is an alternative explanation which is helpful in building our geometric intuition about interpolation. We already mentioned that a genus $2$ quintic curve lies on a quadric. The space of quadrics $\Pbb \H^0(\Oc_{\Pbb^3}(2))$ has dimension $\binom{3+2}{2}-1 = 9$. There is a unique quadric through $9$ general points but none through $10$ general points. In particular, given $10$ general points we should not expect they are all contained in a genus $2$ quintic curve, even though the normal bundle of such a curve has $20 = 2 \cdot 10$ sections.
\end{remark}


\section{The equivalence of strong and regular interpolation}
\label{S:strong-regular}

The goal of this section is to explore the relation between regular and strong interpolation. While it may seem that the latter is strictly stronger than the former, this is not the case. In fact, we already proved the two notions are equivalent for rank $2$ bundles (\cref{T:strong-interpolation-2}). We will extend this result to arbitrary rank.

\begin{theorem}
  \label{T:strong-interpolation}
  Strong and regular interpolation are equivalent.
\end{theorem}

In turn, the equivalence rests on the following strengthening of \cref{T:reducing-lambda}.

\begin{theorem}
  \label{T:interpolation-implication}
  Let $\lambda = (\lambda_i)$ and $\lambda' = (\lambda_i')$ be two tableaux. If $\lambda' \trianglelefteq \lambda$, then $\lambda$-interpolation implies $\lambda'$-interpolation.
\end{theorem}

The heart of the matter lies in the following linear algebra result.

\begin{proposition}
  \label{T:algebra-lambdas}
  Consider a sequence of vector spaces $E_1, \dots, E_m$ and a sequence of integers $\lambda_1, \dots, \lambda_m$. If $\Lambda \subset E_1 \oplus \cdots \oplus E_m$ is a subspace satisfying
  \begin{enumeratea}
  \item
    for each $i$ the projection $\Lambda \rarr E_i$ has rank at least $\lambda_i$, and
  \item
    for each $i$ the projection $\Lambda \rarr \bigoplus_{j \leq i} E_j$ has rank at least $\sum_{j \leq i} \lambda_j$,
  \end{enumeratea}
  then there exists a sequence of subspaces $\Lambda_1, \dots, \Lambda_m \subset \Lambda$ such that
  \begin{enumeratea}
  \item
    $\dim \Lambda_i = \lambda_i$ for all $i$,
  \item
    all projection morphisms $\Lambda_i \rarr E_i$ are injective, and
  \item
    all $\Lambda_i$ are independent.
  \end{enumeratea}
\end{proposition}

\begin{remark}
  Note that the statement of \cref{T:algebra-lambdas} did not impose any order on the $\lambda_i$.
\end{remark}

\begin{proof}
  We proceed by induction on $m$. When $m = 1$, it suffices to pick a subspace $\Lambda_1 \subset \Lambda$ of dimension $\lambda_1$. Next, assume the statement holds for some $m$; we study the statement for $m + 1$. Let $\Lambda'$ be the image of $\Lambda$ in $E_1 \oplus \cdots \oplus E_m$. In order to use the inductive hypotheses hold, we construct spaces $\Lambda_1', \dots, \Lambda_m' \subset \Lambda'$ satisfying the mentioned conditions. For each $1 \leq i \leq m$, we lift $\Lambda_i'$ to $\Lambda_i  \subset \Lambda$ such that $\dim \Lambda_i = \lambda_i$. Since $\Lambda_i'$ are linearly independent, then so are $\Lambda_i$. For each $i$, the projection $\Lambda_i \rarr E_i$ factors through $\Lambda_i \rarr \Lambda_i'$ which is an isomorphism. We know that $\Lambda_i' \rarr E_i$ is injective, hence $\Lambda_i \rarr E_i$ is also injective. It remains to choose $\Lambda_{m+1} \in \Gr(\lambda_{m+1}, \Lambda)$ satisfying
  \[
  \Lambda_{m+1} \cap \left( \bigoplus_{i \leq m} E_i \right) = 0
  \qquad and \qquad
  \Lambda_{m+1} \cap \left( \bigoplus_{i \leq m} \Lambda_i \right) = 0.
  \]
  The first of these conditions implies that the $\Lambda_{m+1}$ intersects trivially the kernel of the projection morphism $\bigoplus_i E_i \rarr E_{m+1}$, so it must be injective onto its image. Together with the fact $\Lambda_1, \dots, \Lambda_m$ are linearly independent, the second condition tells us that by adding $\Lambda_{m+1}$ the list remains linearly independent. That said, we need to argue why such a choice of $\Lambda_{m+1}$ is possible. Both conditions are open, so it suffices to show they describe non-empty varieties in $\Gr(\lambda_{m+1}, \Lambda)$. Numerically, we need to verify
  \[
  \dim \Lambda_{m+1} + \dim \left(\Lambda \cap \bigoplus_{i \leq m} E_i\right) \leq \dim \Lambda
  \qquad and \qquad
  \dim \Lambda_{m+1} + \dim \left(\bigoplus_{i \leq m} \Lambda_i\right) \leq \dim \Lambda
  \]
  If we identify $\Lambda \cap (E_1 \oplus \cdots \oplus E_m)$ with the kernel of the morphism $\Lambda \rarr E_{m+1}$, then the first statement follows from the rank inequality of the morphism in question. The second statement, equivalent to $\lambda_1 + \cdots + \lambda_{m+1} \leq \dim \Lambda$, is a consequence of the lower bound on the rank of $\Lambda \rarr E_1 \oplus \cdots \oplus E_{m+1}$. This completes the proof of desired claim.
\end{proof}

This leads us to a natural criterion for interpolation.

\begin{proposition}
  \label{T:interpolation-criterion}
  A vector bundle $E$ satisfies $\lambda = (\lambda_1, \dots, \lambda_m)$ interpolation if and only if there exist $m$ distinct points $p_1, \dots, p_m$ such that: for every $i$ the evaluation map $\H^0(E) \rarr \bigoplus_{j \leq i} E_{p_j}$ has rank at least $\sum_{j \leq i} \lambda_i$.
\end{proposition}

\begin{proof}
  The forward direction is simple. Let $V_i \subset E_{p_i}$ be a modification datum for $E$ such that $\h^0(M(E, V_i)) = \h^0(E) - \sum \lambda_i$. It follows that the morphism $\H^0(E) \rarr \bigoplus E_{p_i}/V_i$ is surjective; taking a further quotient we deduce that
  \[\xymatrix{
    \H^0(E) \ar[r] & \displaystyle\bigoplus_{j \leq i} E_{p_j}/V_j
  }\]
  is surjective for all $1 \leq i \leq m$. These statements are equivalent to the rank inequalities we desire.

  Conversely, we would like to choose points $p_1, \dots, p_m \in C$ such that
  \begin{enumeratea}
  \item
    \label{T:interpolation-criterion-conditions-a}
    for each $i$ the rank of $\H^0(E) \rarr E_{p_i}$ is at least $\lambda_i$, and
  \item
    \label{T:interpolation-criterion-conditions-b}
    for each $i$ the rank of $\H^0(E) \rarr \bigoplus_{j \leq i} E_{p_j}$ is at least $\sum_{j \leq i} \lambda_j$.
  \end{enumeratea}
  Our hypothesis implies that the open subset of $C^m \setminus \Delta$ cut out by condition (\ref{T:interpolation-criterion-conditions-b}) is non-empty. For $i = 1$ condition (\ref{T:interpolation-criterion-conditions-a}) and (\ref{T:interpolation-criterion-conditions-b}) are equivalent. Since rank is upper semi-continuous and $\lambda_1 \geq \cdots \geq \lambda_m$, it follows that the open specified by condition (\ref{T:interpolation-criterion-conditions-a}) is also non-empty, so we can choose the points $p_i \in C$ as desired.
  
  If $E_i = E_{p_i}$ and $\Lambda$ is the image of $\H^0(E)$ in $\bigoplus E_i$, we can apply \cref{T:algebra-lambdas}. For each $i$, the morphism $\Lambda_i \rarr E_i$ is injective, so a complement $V_i$ to its image would have codimension $\lambda_i$. Since the spaces $\Lambda_i$ are linearly independent, it follows that the evaluation $\H^0(E) \rarr \bigoplus E_i/V_i$ is surjective as desired.
\end{proof}

It is clear that \cref{T:strong-interpolation,T:interpolation-implication} follow from the interpolation criterion in \cref{T:interpolation-criterion}. There are two much simpler corollaries worth spelling out.

\begin{corollary}[Replacing $\lambda$]
  \label{T:replacing-lambda}
  If $\lambda' \trianglelefteq \lambda$, then $(\lambda', \mu) \trianglelefteq (\lambda, \mu)$. In particular, $(\lambda, \mu)$-interpolation implies $(\lambda', \mu)$-interpolation.
\end{corollary}

\begin{corollary}[Breaking $\lambda$]
  If $a = b + c$ is a sum of non-negative integers and $\lambda$ is any tableau, then $(a, \lambda)$-interpolation implies $(b, c, \lambda)$-interpolation.
\end{corollary}

\begin{remark}
  Unlike strong interpolation, weak interpolation is not equivalent its regular counterpart (see \cref{T:interpolation-distinct}). The precise statement is that a vector bundle $E$ satisfying weak interpolation also satisfies regular interpolation if and only if $\h^0(E) \equiv 0,1 \pmod{\rank E}$.
\end{remark}

\bibliographystyle{amsplain.bst}
\bibliography{Interpolation}

\end{document}